\documentclass[a4paper,10pt,reqno]{article}

\usepackage[english]{babel}
\usepackage{amsmath}
\usepackage{amsfonts}
\usepackage{amssymb}
\usepackage{amsthm}
\usepackage{mathrsfs}
\usepackage{float}
\usepackage{graphics}
\usepackage{color}
\usepackage{graphicx}
\usepackage[colorlinks=false]{hyperref}
\usepackage{verbatim}
\usepackage{vmargin}
\usepackage[T1]{fontenc}
\usepackage{times}

\def\e{\varepsilon}
\def\N{{\mathbb N}}
\renewcommand{\to}{\rightarrow}

\usepackage{dsfont}

\numberwithin{equation}{section}

\theoremstyle{plain}
\newtheorem{teor}{Theorem}[section]
\newtheorem{ese}[teor]{Example}
\newtheorem{prop}[teor]{Proposition}
\newtheorem{lem}[teor]{Lemma}
\newtheorem{cor}[teor]{Corollary}

\newcommand{\bdm}{\begin{displaymath}}
\newcommand{\edm}{\end{displaymath}}
\newcommand{\bpb}{\begin{prob}}
\newcommand{\epb}{\end{prob}}

\newcommand{\beq}{\begin{equation}}
\newcommand{\eeq}{\end{equation}}
\newcommand{\bem}{\begin{multline}}
\newcommand{\eem}{\end{multline}}
\newcommand{\bes}{\begin{ese}}
\newcommand{\ees}{\end{ese}}
\newcommand{\bde}{\begin{defi}}
\newcommand{\ede}{\end{defi}}
\newcommand{\bpr}{\begin{prop}}
\newcommand{\epr}{\end{prop}}
\newcommand{\ble}{\begin{lem}}
\newcommand{\ele}{\end{lem}}
\newcommand{\bte}{\begin{teor}}
\newcommand{\ete}{\end{teor}}
\newcommand{\bco}{\begin{cor}}
\newcommand{\eco}{\end{cor}}

\theoremstyle{definition}
\newtheorem{defi}[teor]{Definition}
\newtheorem{remark}[teor]{Remark}
\usepackage{geometry}
\newcommand{\R}{\mathbb{R}}

\newcommand{\T}{\mathbb{T}}

\newcommand{\calO}{{\mathcal O}}
\newcommand{\calP}{{\mathcal P}}

\newcommand{\al}{\alpha}
\newcommand{\ka}{\kappa}
\newcommand{\s}{\sigma}

\newcommand{\del}{\partial}

\newcommand{\oo}{\omega}

\newcommand{{\resonance}}{relevant self-energy cluster }

\begin{document}

\title{\bf On the integrability of Degasperis-Procesi equation: \\
control of the Sobolev norms and Birkhoff resonances.
}
\date{}

\author{\bf 
Roberto Feola$^{**}$, Filippo Giuliani$^{\dag}$, Stefano Pasquali$^{\dag}$ 
\\
\small
${}^{**}$ SISSA, Trieste, rfeola@sissa.it; 
\\
\small
${}^\dag$ RomaTre, Roma, fgiuliani@mat.uniroma3.it, spasquali@mat.uniroma3.it\footnote{
This research was supported by PRIN 2015 ``Variational methods, with applications to problems in 
mathematical physics and geometry''.
This research was supported by PRIN 2012 ``Variational and perturbative aspects of nonlinear differential problems''.
This research was supported by by ERC grant 306414 HamPDEs under FP7.
}
}

\maketitle

\begin{abstract}
We consider the \emph{dispersive} Degasperis-Procesi equation $u_t-u_{x x t}-\mathtt{c} u_{xxx}+4 \mathtt{c} u_x-u u_{xxx}-3 u_x u_{xx}+4 u u_x=0$
with $\mathtt{c} \in \mathbb{R} \setminus \{0\}$. In \cite{Deg} the authors proved that this equation possesses infinitely many conserved quantities. We prove that there are infinitely many of such constants of motion which control the Sobolev norms and which are analytic in a neighborhood of the origin of the Sobolev space $H^s$ with $s \geq 2$, both on $\mathbb{R}$ and $\T$. By the analysis of these conserved quantities we deduce a result of global well-posedness for solutions with small initial data and we show that, on the circle, the \emph{formal} Birkhoff normal form of the Degasperis-Procesi at any order is action-preserving.

%
%
\end{abstract}

\tableofcontents

\section{Introduction}

In $1999$ Degasperis and Procesi \cite{DegPro} applied a method of asymptotic integrability to the family of third-order dispersive PDE conservation laws
\begin{equation}\label{DPfamiglia}
u_t+c_0 u_x+\gamma u_{xxx}-\alpha^2 u_{xx t}=(c_1 u^2+c_2 u_x^2+c_3 u u_{xx})_x,
\end{equation}
where $\alpha, c_0, c_1, c_2, c_3$ are real constants and subindices denote the partial derivatives.\\
Only three equations within this family result to satisfy the asymptotic integrability condition, the KdV equation ($\alpha=c_2=c_3=0$), the Camassa-Holm equation ($c_1=-3 c_3/2 \alpha^2, c_2=c_3/2$) and the Degaperis-Procesi equation
\begin{equation}\label{DPvera}
u_t+c_0 u_x+\gamma u_{xxx}-\alpha^2 u_{xx t}=-\frac{4 c_2}{\alpha^2} u u_x+ 3 c_2 u_x u_{xx}+ c_2 u u_{xxx}.
\end{equation}
In \cite{Deg} Degasperis-Holm-Hone showed that the equation \eqref{DPvera} is \emph{integrable} by constructing its Lax pair. They also proposed a bi-Hamiltonian structure and a recursive method to generate infinitely many constants of motion (see Section $4$ in \cite{Deg}).\\
Later Constantin and Lannes showed in \cite{ConstLannes} that the Degasperis-Procesi equation, as well as the Camassa-Holm equation, can be regarded as a model for nonlinear shallow water dynamics, and that it accomodates wave-breaking phenomena.

\medskip

We observe that the equation \eqref{DPvera} is a quasi-linear PDE, namely the linear and the nonlinear terms contain the same order of derivatives. We also remark that \eqref{DPvera} is not translation invariant and the linear dispersion (see for instance \eqref{dispersionLaw} for the case $x\in\T$) is related to the chosen frame.\\
By taking $c_0=-\gamma$ and $\alpha^2=1$ in \eqref{DPvera} the linearized equation at $u=0$ transforms into the following transport equation
\begin{equation}
u_t=\gamma u_x.
\end{equation}
Hence, in this case, all the solutions are travelling waves and there are no dispersive effects. In particular, by choosing $c_0=\gamma=0$, $c_2=1$ and $\alpha^2=1$, the equation \eqref{DPvera} can be transformed into the \emph{dispersionless} form
\begin{equation}\label{dispersionless}
u_t-u_{x x t}+ 4 u u_x=3 u_x u_{xx}+u u_{xxx}.
\end{equation}

The family of equations \eqref{DPvera} is covariant under the group of transformations
 \begin{equation}\label{boosts}
u\mapsto \lambda u( t, \xi x+\eta t ) +\delta, \qquad \lambda, \xi, \eta, \delta\in\mathbb{R},
\end{equation}
which are compositions of translations and Galilean boosts, and 
all the parametrized equations in \eqref{DPvera} can be obtained from \eqref{dispersionless} by applying such changes of coordinates (see \cite{scolar}).
As we said above, in order to consider the dispersive effects of \eqref{DPvera} we have to impose that $c_0\neq -\gamma$ and $c_0, \gamma\neq 0$; if we let the coefficient in front of $u_{xx t}$ be $-1$, we can obtain an equation with this feature from \eqref{dispersionless} if and only if we apply a transformation of the form \eqref{boosts} with $\delta\neq 0$, namely we have to consider translations of the variable $u$.\\
We will consider the \emph{dispersive} Degasperis-Procesi equation
\begin{equation}\label{DP}
u_t-u_{x x t}-\mathtt{c} u_{xxx}+4 \mathtt{c} u_x- u u_{xxx}-3 u_x u_{xx}+4  u u_x=0,
\end{equation}
obtained from \eqref{dispersionless} by translating $u\mapsto u+\mathtt{c}$, for some real parameter $\mathtt{c}\neq 0$. We remark that the same equation can be obtained from \eqref{DPvera} by setting $
\alpha^2=1, \gamma=-\mathtt{c},  c_0=4 \mathtt{c}, c_2=c_3=1.
$

\smallskip

We note that the mass $\int u \,dx$ is a constant of motion of the Degasperis-Procesi equation \eqref{dispersionless}, hence the subsets of functions with fixed finite average
\begin{equation}\label{polpetta}
\mathfrak{L}_{\mathtt{c}}:=\{ u : \int u\,dx=\mathtt{c}\}
\end{equation}
are left invariant by the flow of \eqref{dispersionless}. On the subspace $\mathfrak{L}_0$ it is possible to define a non-degenerate symplectic structure for which \eqref{dispersionless} can be seen as a Hamiltonian PDE. Our purpose is to investigate the dynamics of \eqref{dispersionless} on the invariant subsets $\mathfrak{L}_{\mathtt{c}}$ with $\mathtt{c}\neq 0$, which is equivalent to the one of \eqref{DP} when $u \in \mathfrak{L}_0$.\\

The equation \eqref{dispersionless} has been widely investigated by many authors since it presents wave breaking phenomena, peakon and soliton solutions and blow-up scenarios (see for instance \cite{ConstantinEscher}, \cite{Escher}, \cite{CH}, \cite{stabpeak}, \cite{Matsuno}, \cite{Li}, \cite{Liu}, \cite{LiuYin}, \cite{Deg} and \cite{globbreak}). \\
Lundmark and Szmigielski \cite{Lund} presented an inverse scattering approach for computing n-peakon solutions to equation \eqref{dispersionless}. Vakhnenko and Parkes \cite{Vak} investigated traveling wave solutions. Holm and Staley \cite{Ho} studied stability of solitons and peakons numerically.\\
Regarding the well-posedness of the equation \eqref{dispersionless} we cite the results of Yin \cite{YinR}, \cite{YinT} on the local well-posedness with initial data $u_0\in H^s$, $s>3/2$, both on the line and on the circle, and the result by Himonas-Holliman \cite{HH}, who showed that for \eqref{dispersionless} the data-to-solution map is not uniformly continuous. We also mention the paper of Coclite-Karlsen \cite{Coclite} for the well-posedness in classes of discontinuous functions and Wu \cite{Wu} for the periodic generalized Degasperis-Procesi equation.

\medskip

Although the bi-Hamiltonian structure of equation \eqref{dispersionless} provides an infinite number of conservation laws (\cite{Deg}), there is no way to find conservation laws controlling the $H^1$-norm of $u \in \mathfrak{L}_0$ (\cite{Escher}), which for instance represents an important difference with respect to the Camassa-Holm equation. We want to show that if we consider the dynamics on $\mathfrak{L}_{\mathtt{c}}$ with $\mathtt{c}\neq 0$, then the situation, close to the origin, is truly different.



\medskip

In this paper we construct infinitely many conserved quantities $K_n$ for the dispersive Degasperis-Procesi equation \eqref{DP}, starting from the ones proposed by Degasperis-Holm-Hone in \cite{Deg}.\\
We prove that the $K_n$'s are analytic functions in a ball centered at the origin of $H^{n+1}$ with radius depending on the parameter $\mathtt{c}$, but independent of $n$ (see \eqref{DP}), such that
\begin{itemize}
\item[$(F1)$] for any $s \geq 2$ there exists $n=n(s)\in\mathbb{N}$ such that $K_n$ controls the $H^s$-norm of solutions $u$ of \eqref{DP} with initial datum $u_0$ sufficiently close to $0$ in $H^s$,
\item[$(F2)$] the quadratic part of $K_n(u)$ is given by
\[
K_n^{(0)}=\int (\partial_x^{n-1} w)^2\,dx, \quad w:=u-u_{xx}
\]
and so, if $x\in\T$, it reads
\[
K_n^{(0)}=\sum_{j\in\mathbb{Z}} (1+j^2)^2 \,j^{2 (n-1)}\,\lvert u_j \rvert^2.
\]
\end{itemize}
These facts are actually proved in Theorem \ref{costantiMotoDP}, which is the main result of the paper. \\
We provide also two applications of this result. The first one is a result of global in time well-posedness and stability in a neighborhood of the origin. 

\begin{itemize}
\item[$(A1)$] Fixed $s \geq 2$, there exists a ball centered at the origin $B_s$ of $H^s$ in which the equation \eqref{DP} is \emph{ globally well-posed} (see \cite{Tao}), namely any solution with initial datum belonging to $B_s$ exists for all times and, moreover, it remains inside a slightly bigger ball centered at the origin (hence the origin is a stable fixed point).
\end{itemize}
This is the content of Theorem \ref{esistenzaglobale}, which in turn is due to $(F1)$. \\
The use of the conserved quantities to prove existence results for solutions of integrable PDE's is a classical argument. We mention for instance the celebrated result by Bona-Smith \cite{BonaSmith} on the initial value problem for the KdV equation.\\
We also mention \cite{Luca}, \cite{Visciglia}, \cite{Visciglia2}, in which the properties of  conserved quantities of other integrable PDE's, such as the DNLS and the Benjamin-Ono equation on $\T$, are used to construct functional Gibbs measures and study the long time existence of regular solutions. It would be interesting to investigate if similar results could be applied to equation \eqref{DP}. \\


\medskip

Another application of Theorem \ref{costantiMotoDP} 
concerns the Birkhoff normal form analysis for the Degasperis-Procesi equation on the circle. First we briefly recall some facts and literature about Birkhoff normal form.

\smallskip

The Birkhoff normal form theory for PDE's has been widely developed since the 1990s, in order to extend to the infinite dimensional dynamical case the classical results which hold for finite dimensional ``nearly'' integrable Hamiltonian systems.

Let us consider a Hamiltonian function 
$H(p,q)=H=H^{(0)}+P$, with $(p,q)\in \R^{2n}$
with
\begin{equation}\label{ham}
H^{(0)}=\sum_{j=1}^{n}\oo_j\frac{p_j^{2}+q_{j}^{2}}{2},
\end{equation}
and $P$ is a smooth function with a zero of order at least $3$ at the origin. The origin is an elliptic fixed point and the Hamiltonian system can be seen as a system of harmonic oscillators with frequencies $\omega_j$ coupled by the nonlinearities.
Classical theory guarantees that, for any $N\geq1$, there is a real analytic and symplectic map  $\Phi_{N}$ 
such that 
\[
H\circ{\Phi_{N}}=H^{(0)}+{Z}_{N}+{R}_{N},
\]
where ${R}_{N}$ is a function with a zero of order $N+3$ and ${Z}_{N}$ is a polynomial of degree $N+2$
which Poisson-commutes with $H^{(0)}$. In particular, if the frequencies $\oo_{j}$ are ``non resonant'', namely
 \begin{equation}\label{reso}
\sum_{j=1}^{n}\oo_{j}k_{j}\neq 0, \quad \forall\; k\in \mathbb{Z}^{n}\setminus\{0\},
 \end{equation}
then ${Z}_{N}$ depends only on the actions $I_{j}=(p_{j}^{2}+q_{j}^{2})/2$.
This implies that
if the initial datum has size $\e\ll1$ then the solution remains in a neighborhood of radius $2\e$
for times of order $\e^{-(N+1)}$.  The key idea to obtain such a result is to remove from the nonlinearity $P$ all the monomials 
which do not commutes with $H^{(0)}$. This can be done iteratively by means of symplectic transformations. More precisely,
one uses a sequence of maps $\Phi_{p}$ generated as the Hamiltonian flow at time one of an auxiliary Hamiltonian $F_{p}$ of degree
of homogeneity $p+2$. The $F_{p}$ is chosen in such a way the 
 equation 
 \begin{equation}\label{homo}
 \{H^{(0)},F_p\}+G_{p}=Z
 \end{equation}
 where $G_p$ is some homogeneous Hamiltonian of degree $p+2$ and $\{H^{(0)},Z\}=0$, is satisfied.
 It turns out that, in order to solve it , one needs some non-resonance conditions on the frequencies $\oo_{j}$, for instance the relation \eqref{reso}.

In the case $n<\infty$ the number of monomials which have to be canceled out is finite. 
Many PDE's (NLS, KdV, Klein-Gordon...) $u_t=L(u)+f(u)$, with $L$ possibly an unbounded linear operator and $f$ some nonlinear function, 
can be written, on compact manifolds, as Hamiltonian systems whose quadratic part has a form similar to \eqref{ham} with $n=\infty$. \\
There are several difficulties in extending the theory to the infinite dimensional case:

\begin{enumerate}
\item[(i)] one needs suitable \emph{non-resonance conditions} which replace \eqref{reso} in infinite dimension;
\item[(ii)] one needs to cancel out an \emph{infinite} number of monomials;
\item[(iii)] one needs to check that
the normal form $Z_{N}$ is \emph{action-preserving};
\item[(iv)] the Birkhoff transformations are flows of possibly ill-posed PDEs.
\end{enumerate}

The first difficulty  has been overcome  in \cite{Bambu} and in \cite{BG}.
To deal with the second one, a good definition of the class of formal polynomials on which one works is required. 
For instance, one at least needs that the Hamiltonian $F_p$ are continuos functions on the phase space. 
In the framework of Hamiltonian PDEs typical phase spaces are
Sobolev spaces of functions (or weighted spaces of sequences). We mention \cite{DelortSzeft2} and \cite{DelortSzeft1}, in which the authors introduce a suitable classes of multilinear forms. \\
Concerning item $(iii)$, we say that the normal form $Z_k$ is \textbf{action-preserving} if it depends only on the actions $\lvert u_j \rvert^2$.  This implies that the flow generated by the Hamiltonian function $Z_{k}$ leave the actions invariant.
In the PDE context this means that the Sobolev norms $H^{s}$ are preserved for any $s>0$. Proving that the obtained normal form is action-preserving is a problem concerning the specific equation one is studying. \\
Regarding item $(iv)$, if the nonlinearity depends upon some derivatives the Birkhoff transformations could be not well-defined. Especially for quasi-linear PDEs, the problem of constructing rigorous symplectic transformations in order to apply a Birkhoff normal form procedure is delicate.\\
The Birkhoff normal form methods have been used by many authors to prove long time existence of solutions with data in a small neighborhood of a fixed point. 
We quote for instance the papers by Delort  \cite{Delort-2009}, \cite{Delort-sphere}, in which the author studied  
 quasi-linear perturbations of Klein-Gordon (K-G)
and the recent paper by Berti-Delort  \cite{BertiDelort}, where the authors 
applied a suitable Birkhoff normal form procedure to the capillarity-gravity water waves (WW) equation. 

In the latter papers  the linear frequencies depends
on some  ``external'' parameters (the ``mass'' for the K-G, the capillarity for the WW). This fact has been used 
in order to prove very strong non-resonance conditions (which hold for ``most'' values of parameters) 
and, as a consequence, that  the normal forms are action-preserving.\\
When the equation does not depend on physical parameters, the problem of showing the integrability of the normal form relies on an analysis of the algebraic structure of the resonances.
We mention the paper by Craig-Worfolk
\cite{CraigBF} in which the authors study (at a formal level) 
the Birkhoff normal form for 
pure-gravity water waves at order four. It is known that for this equation 
there are non trivial resonances (called ``Benjamin-Feir''), 
which could prevent the normal forms to be  action-preserving.
Actually, they show that 
there are suitable cancellations in the coefficients of the Hamiltonian which allows to obtain 
an action preserving normal form.\\
Our result focuses only on $(iii)$, by formulating the Birkhoff normal form procedure only at a formal level, and it is similar to the study performed in \cite{CraigBF}, since we do not deal with $(i)$ , $(ii)$ and $(iv)$ (which concern convergence problem).\\
The main differences are that:
\begin{itemize}
\item we use the integrability of the equation in order to overcome the problem of non trivial resonances;
\item we are able to prove the integrability of the normal form at any order, since we exploit the algebraic structure of the resonances.
\end{itemize} 
More precisely we prove the following: 
\begin{itemize}
\item[$(A2)$] 
at a purely formal level, it is possible to put the Hamiltonian \eqref{DPHamiltonian} in a action-preserving Birkhoff normal form at any order. 
\end{itemize}
This result is achieved thanks to $(F2)$ and it is the content of Theorem \ref{teoBirk}.

\vspace{0.4em}
One of the main applications of the Birkhoff normal form methods concerns the KAM theory for PDEs. It is well known that in order to apply perturbative arguments to construct periodic and quasi-periodic solutions for perturbed autonomous integrable equations one needs to control the frequencies of the expected invariant tori. In the infinite dimensional context, this requires the presence of parameters which modulate the frequencies, since the non-resonance conditions to be imposed are quite complicated. When the considered system does not present external parameters, one has to extract them from the equation itself: a way to do that is to perform a Birkhoff normal form procedure. Actually, the result presented in Theorem \ref{teoBirk} has been motivated by the study of quasi-periodic solutions for perturbations of the Degasperis-Procesi equation \cite{FGP}.

\subsection{Preliminaries}\label{prelimi}

In order to state the main result of the paper we introduce the Hamiltonian setting and the space of formal polynomials and power series. When there is no specification of the spatial domain we mean that the arguments hold for both cases $x\in \T$ or $x\in\mathbb{R}$.

\paragraph{Hamiltonian setting.}
The equation \eqref{DP} can be formulated as a Hamiltonian PDE $u_t=J\,\nabla_{L^2} H(u)$, where $\nabla_{L^2} H$ is the $L^2$ gradient of the Hamiltonian
\begin{equation}\label{DPHamiltonian}
H(u)=\int \frac{\mathtt{c}\,u^2}{2}-\frac{u^3}{6}\, dx.
\end{equation}
The Hamiltonian \eqref{DPHamiltonian} is defined on the real phase space (recall \eqref{polpetta})
\begin{equation*}
H_0^1:=H^1\cap \mathfrak{L}_0
\end{equation*}
endowed with the non-degenerate symplectic form
\begin{equation}\label{SymplecticFormDP}
\Omega(u, v):=\int (J^{-1} u)\,v\,dx, \quad \forall u, v\in H_0^1, \qquad J:=(1-\partial_{xx})^{-1}(4-\partial_{xx})\partial_x.
\end{equation}


The Poisson bracket induced by $\Omega$ between two functions $F, G\colon H_0^1 \to \mathbb{R}$ is
\begin{equation}\label{PoissonBracketDP}
\{ F(u), G(u) \}:=\Omega(X_F, X_G)=\int \nabla F(u)\,J \nabla G(u)\,dx,
\end{equation}
where $X_F$ and $X_G$ are the vector fields associated to the Hamiltonians $F$ and $G$, respectively.\\
On the circle the \textit{dispersion law} of the Degasperis-Procesi equation is given by 
\begin{equation}\label{dispersionLaw}
j\mapsto\omega(j):=j\,\frac{4+j^2}{1+j^2}=j+\frac{3 j}{1+j^2}, \qquad j\in\mathbb{Z}\setminus\{0\},
\end{equation}
where $\omega(j)$ are the \emph{linear frequencies of oscillations} or the eigenvalues of the operator $J$ on the circle (see \eqref{SymplecticFormDP}).

\vspace{0.7em}

Let us define
 \begin{equation}\label{emme}
 w:=(1-\del_{xx})u, \quad m:=\mathtt{c}+u-u_{xx},\quad p=-m^{\frac{1}{3}}.
 \end{equation}
One can easily check that
the functions 
 \begin{equation}\label{nonleho}
 \begin{aligned}
 &H(u)=\int \frac{\mathtt{c}\,u^2}{2}-\frac{u^3}{6}\, dx,\quad M_0(u)=\frac{1}{2}\int (J^{-1}u_x)u dx,\quad M_1(u)=\int m^{\frac{1}{3}}dx,
 \end{aligned}
 \end{equation}
are constant of motions for equation \eqref{DP}, i.e. if $u(t,x)$ solves \eqref{DP} then
\begin{equation}\label{pissonCOMMU}
\frac{d}{dt}M_0(u)=\{ M_0, H\}(u)=0, \quad \frac{d}{dt}M_1(u)=\{ M_1, H\}(u)=0.
\end{equation}
We will consider the Sobolev spaces
\begin{equation}\label{spazio}
H^{s}(\mathbb{T};\mathbb{R}):=
\big\{ u(x)\in H^1_0(\mathbb{T};\mathbb{R}) : \|u\|_{H^{s}}^2:=\sum_{j\in\mathbb{Z} \setminus \{0\} } \lvert u_{j} \rvert^2 \langle j \rangle^{2 s}<\infty, \,\,\overline{u}_j=u_{-j} \big\}
\end{equation}
where $\langle j\rangle:=\sqrt{1+j^{2}}$ and
\begin{equation}\label{spazioR}
H^{s}(\mathbb{R};\mathbb{R}):=
\big\{ u(x)\in H^1_0(\mathbb{R};\mathbb{R}) : \|u\|_{H^{s}}^2:=\sum_{k=0}^s \int_{\mathbb{R}} (\partial_x^k u)^2\,dx \big\}.
\end{equation}
We will denote both the spaces \eqref{spazio} and \eqref{spazioR} with $H^s$ in Section \ref{costantiMotoPROOF}, since all the arguments hold independently by the $x$-space.
We denote by
\begin{equation}
\lvert u \rvert_{L^{\infty}}:=\sup_x \lvert u(x) \rvert
\end{equation}
the $L^{\infty}$-norm either on $\mathbb{R}$ or on $\mathbb{T}$.\\
Given a Banach space $(E, \lVert \cdot \rVert_E)$ and $r\geq 0$, we denote by
\[
B_{E}(v, r):=\{ u\in E : \lVert u-v \rVert_{E}< r  \}
\]
the open ball centered at $v\in E$ with radius $r$.
\paragraph{Space of polynomials.}
When $x \in \mathbb{T}$ it is convenient to introduce a class of polynomials which describes the Hamiltonians in terms of their Fourier coefficients. These definitions will be used in Section \ref{proofteoBirk}.\\ 
We use the multi-index notation $\alpha\in \mathbb{N}^{\mathbb{Z}}$, $\lvert \alpha \rvert:=\sum_j \alpha_j$.
We define
\begin{itemize}
\item the monomial associated to $\alpha$: $u^{\alpha}:=\prod_j u_j^{\alpha_j}.$
\item the momentum associated to $\alpha$: $\mathcal{M}(\alpha)=\sum_j j\, \alpha_j$.
\item the divisor associated to $\alpha$ (recall the linear frequencies \eqref{dispersionLaw}): $\Omega(\alpha):=\sum_j \omega(j)\alpha_j$.
\end{itemize} 
We define the set of indices with zero momentum of order $n\in\mathbb{N}$
\begin{equation}
\mathcal{I}_n:=\{ \alpha\in\mathbb{N}^{\mathbb{Z}} : \lvert \alpha \rvert=n+2, \,\,\mathcal{M}(\alpha)=0 \}.
\end{equation}

\begin{defi}\label{defpolinomi}
We say that $P:= \big( P_{\alpha} \big)_{\alpha\in\mathcal{I}_n}$, $P_{\alpha}\in\mathbb{C}$ for any $\alpha\in\mathcal{I}_n$, is a \emph{formal homogenous polynomial of degree $n+2$} and we write the (formal) expression $P(u)=\sum_{\alpha\in\mathcal{I}_n} P_{\alpha} u^{\alpha}$.\\
We call $\mathscr{P}^{(n)}$ the space of the formal homogenous polynomial of degree $n$.
\end{defi}
\begin{defi}\label{defserie}
We define the space product
\[
\mathscr{F}:=\prod_{n\geq 0}\mathscr{P}^{(n)}.
\]
If $P\in \mathscr{F}$ then we write the (formal) expression $P=\sum_{n=0}^{\infty} P^{(n)}$ where $P^{(n)}\in \mathscr{P}^{(n)}$.
\end{defi}
There exists a obvious inclusion of $\mathscr{P}^{(n)}$ into $\mathscr{F}$ given by
\[
(P_{\alpha})_{\alpha\in\mathcal{I}_n} \mapsto (\underbrace{0, \dots, 0}_{(n-1) \mbox{times}}, (P_{\alpha})_{\alpha\in\mathcal{I}_n}, 0, \dots)
\]
and we denote by $\Pi^{(n)}\colon \mathscr{F}\to \mathscr{P}^{(n)}$ the projection
\[
(\dots, (P_{\alpha})_{\alpha\in\mathcal{I}_n}, \dots)\mapsto (P_{\alpha})_{\alpha\in\mathcal{I}_n}.
\]

\smallskip

We call $\mathscr{P}^{(\le n)}:=\prod_{k=0}^n \mathscr{P}^{(k)}$ the space of the formal polynomials of degree (at most) $n+2$. As above, $\mathscr{P}^{(\le n)}$ can be embedded into the space of formal power series $\mathscr{F}$. We denote by $\Pi^{(\le n)}$ the projection of $\mathscr{F}$ onto $\mathscr{P}^{(\le n)}$.\\
We define $\mathscr{F}^{\geq n}:=\prod_{k\geq n}\mathscr{P}^{(k)}$. We denote by $\Pi^{(\geq n)}$ the projection of $\mathscr{F}$ onto $\mathscr{P}^{(\geq n)}$.

In particular, if $G$ is a formal power series we write
\[
\Pi^{(n)} G=G^{(n)}, \quad \Pi^{(\le n)} G=G^{(\le n)}, \quad \Pi^{(\geq n)} G=G^{(\geq n)}.
\]

\paragraph{Birkhoff resonances.}
Now we introduce the notion of Birkhoff resonances.

\begin{defi}\label{defrisonanza}
We say that $\alpha\in \mathcal{I}_n$ is \textbf{resonant} if its associated divisor $\Omega(\alpha)=0$ and we write $\alpha\in\mathcal{N}_n$.\\
We say that $\alpha$ is \textbf{trivially resonant} if $\alpha_j=\alpha_{-j}$ for all $j$ and we write $\alpha\in\mathcal{N}^*_n$.
By the fact that $\omega(-j)=-\omega(j)$, if $\alpha$ is trivially resonant then it is resonant and its associated monomials depend only on the \textbf{actions} $I_j:=\lvert u_j \rvert^2=u_j u_{-j}$ 
\[
u^{\alpha}=\prod_{j>0} (\lvert u_{j} \rvert^2)^{\alpha_{j}}=\prod_{j>0} I_{j}^{\alpha_{j}}.
\]
We say that a polynomial which depends only upon the actions $I_j$ is \textbf{action-preserving}.
\end{defi}
\subsection{Main result and applications}

The main result of the paper is the following.

\begin{teor}[{\bf Constants of Motion}]\label{costantiMotoDP}
Let $\mathtt{c}\in\mathbb{R}\setminus\{ 0\}$. For any $n\geq1$ there exist a decreasing sequence of positive numbers $(r_n)_{n \geq 1}$, and a sequence of functions $K_{n}: H^{n+1} \to \R$ with the following properties:

\begin{itemize}

\item[(0)] {\bf Involution:} if we set $K_0:=H$ (see \eqref{DPHamiltonian}) then for any $n\geq0$ one has that 
\begin{equation}\label{commutano}
\{H(u),K_{n}(u)\}=0.
\end{equation}

\item[(i)] {\bf Analyticity:} the function $K_n$ is analytic on $B_{H^{n+1}}(0, \lvert \mathtt{c} \rvert /2)$; more precisely, there exists a function $\Psi_n\colon\mathbb{C}^{n+2}\to \mathbb{C}$ analytic on $B_{\mathbb{C}}(0,  \lvert \mathtt{c} \rvert)\times\dots\times B_{\mathbb{C}}(0,  \lvert \mathtt{c} \rvert)$ such that
\[
K_n(u)=\int \Psi_n(u, u_x, \dots, \partial_x^{n+1} u)\,dx.
\]
Moreover, $\Psi_n$ admits the following Taylor expansion
\begin{equation}\label{formaaffine}
 \Psi_n(u, u_x, \dots, \partial_x^{n+1} u)=\sum_{k\geq 0}\,\, \sum_{\substack{\alpha\in \mathbb{N}^{\{0, \dots, n+1\}}, \\   \sum \alpha_i=k,\\ \sum i\alpha_i \le n+1  } } \Psi_{\alpha} \,\, u^{\alpha_0}( \partial_x u)^{\alpha_1}\dots ( \partial_x^{n+1} u)^{\alpha_{n+1}}, \qquad \Psi_{\alpha}\in \mathbb{C}.
\end{equation}
\item[(ii)] {\bf Characterization of quadratic parts:} the Taylor polynomial of order $2$ of $K_n$ at $u=0$ has the form
\begin{equation}\label{formacostanti}
\begin{aligned}
&K_n^{(0)}(u)=\int_{\mathbb{R}} (\partial_x^{n-1} (u-u_{xx}))^2\,dx \quad x\in\mathbb{R}, \\
& K_{n}^{(0)}(u)=
\sum_{j\in \mathbb{Z}\backslash\{0\}}|j|^{2(n-1)}(1+j^{2})^2|u_{j}|^{2} \quad x\in\T; 
\end{aligned}
\end{equation}

\item[(iii)] {\bf Control of Sobolev norms:} there exist positive constants $C=C(n, \mathtt{c})$ and $\tilde{c}=\tilde{c}(n, \mathtt{c})$ such that for any $n\geq1$
\begin{equation}\label{equivalenzaNorma}
|K_{n}^{(0)}(u)|\leq \|u\|^2_{H^{n+1}}\leq \|u\|^2_{L^{2}} 
+\tilde{c} |K_{n}^{(0)}(u)| \quad \forall u\in B_{H^{n+1}}(0,  \lvert \mathtt{c} \rvert/2)
\end{equation}
and  
\begin{equation}\label{stimecostanti}
|K_{n}^{(\geq1)}(u)|\leq C\|u\|^{3}_{H^{n+1}}  \quad \forall u\in  B_{H^{n+1}}(0, r_n).
\end{equation}

\end{itemize}

\end{teor}

Let us make some comments.
\begin{itemize}

\item 
\eqref{equivalenzaNorma} and \eqref{stimecostanti} imply that $K_n$ is equivalent to the $H^{n+1}$-norm in a neighborhood of the origin, and this does not hold as the parameter $\mathtt{c}$ goes to zero (see for instance Remark \ref{clinprop} and Remark \ref{displess}). 

\item We remark that the radius of analyticity of the $K_n$'s depends only on the parameter $\mathtt{c}$.

\item Our result is based on an explicit computation of the coefficients of the quadratic part of the constructed conserved quantities.
We point out that the radii $r_n$ in \eqref{stimecostanti} decrease to zero as $n\to \infty$.
 It may be possible to improve \eqref{stimecostanti} by studying the higher order expansions of the constants of motion.
 
 \item By the expression \eqref{formaaffine} the function $\Psi_n$ is affine in the variable $\partial_x^{n+1} u$ (see Remark \ref{remarkaffine}). This is a key point to prove the bounds in item $(iii)$.

\end{itemize}

Let us discuss the applications we obtain  by the result above. 
\vspace{0.8em}

%

\noindent
We prove the following stability result.
\begin{teor}[{\bf Stability and Global existence}]\label{esistenzaglobale}
Let $X$ be $\mathbb{R}$ or $\T$. For any $s \geq 2$ there is $r=r(s)>0$
such that for any $u_0\in B_{H^s}(0, r)$
there exists a unique solution $u(t,x)$ of \eqref{DP}, defined for all times, belonging to $C^{0}(\mathbb{R};H^{s}(X;\mathbb{R}))$
 such that
 \[
 \sup_{t\in\R}\|u(t,x)\|_{H^{s}}\leq C' r,
 \]
 for some constant $C'=C'(s, \mathtt{c})>0$.
\end{teor}

The above theorem is in turn based on a local well-posedness result for the equation \eqref{DP} (the proof, which follows \cite{HH}, is deferred to the Appendix).

\medskip

The second application concerns the study of the Birkhoff normal form of the equation \eqref{DP}. 

\begin{teor}[{\bf Formal Birkhoff normal form}]\label{teoBirk}
Let $H$ be the Hamiltonian \eqref{DPHamiltonian} with $x\in\T$. For any $N\in\mathbb{N}$ there exist, at least formally, a symplectic transformation $\Phi_N$ such that
\begin{equation}\label{BirkhoffFormN}
H\circ \Phi_N=H^{(0)}+Z_N+R_N
\end{equation}
where $Z_N\in \mathscr{P}^{(\le N)}$ (recall Definition \ref{defpolinomi}) is action-preserving, hence it Poisson commutes with $H^{(0)}$ and depends only on the \emph{actions} $I_j:=\lvert u_j \rvert^2$. The function $R_N\in\mathscr{F}^{(\geq N+1)}$ (recall Definition \ref{defserie}).
\end{teor}
\begin{defi}\label{defBirkformN}
We say that a Hamiltonian $G\in\mathscr{F}$ is in a Birkhoff normal form of order $N$ if it has the form \eqref{BirkhoffFormN} described in Theorem \ref{teoBirk}.
\end{defi}

In the proof of such result will be fundamental the explicit form of the quadratic part of the constant of motion that we give in
\eqref{formacostanti}. 
The proof of Theorem \ref{teoBirk} is based on the following classical result (see, for instance, \cite{KdVeKAM}):
\begin{itemize}
\item
 two commuting Hamiltonians $H,K\in\mathscr{F}$ can be put in Birkhoff normal form, up to order $N$, by the same change of coordinates
 (at least at the formal level).
\end{itemize}
This fact will be proved in Lemma \ref{lemmabellissimo}, which is a variation of Theorem $G. 2$ in \cite{KdVeKAM}, since we do not assume that the linear frequencies are non resonant.

\paragraph{Plan of the paper} The paper is organized as follows. In Section \ref{costantiMotoPROOF} we prove Theorem \ref{costantiMotoDP}.
 In Section \ref{GWPDP} we first state a local-well posedness result, which is proved in  the Appendix \ref{appendA}, and then we prove \ref{esistenzaglobale} by using the bounds \eqref{equivalenzaNorma}, \eqref{stimecostanti} and a bootstrap argument. In Section \ref{proofteoBirk} we give a proof of Theorem \ref{teoBirk}. 
 
\paragraph{Acknowledgements} We warmly thank Michela Procesi, Luca Biasco and Alberto Maspero for many useful suggestions and fruitful discussions.

\section{Constants of motion}\label{costantiMotoPROOF}

In Section $3$ of \cite{Deg} Degasperis-Holm-Hone give the Lax pair for the equation \eqref{DPfamiglia}, which we write in the following with the choice of the parameters that leads to consider the equation \eqref{DP} (recall the definition of $m$ in \eqref{emme}),
\begin{equation}\label{LaxPair}
\begin{cases}
(1-\partial_{xx})\Psi_x &= m\Psi \\
\Psi_t+\frac{1}{\lambda}\Psi_{xx}+(u+\mathtt{c})\Psi_{x}-u_x \Psi &= 0
\end{cases}
,
\end{equation}
for a real parameter $\lambda\neq 0$.
In Section $4$ of \cite{Deg} the authors derive many conservation laws by considering the following relations,
 \begin{equation}\label{RELAZIONE}
 (1-\del_{xx})\rho=3\rho \rho_x+\rho^{3}+\lambda m,
 \end{equation}
 and
\begin{equation}\label{conservationlaw}
\rho_t=j_x, \qquad j=u_x-\frac{1}{\lambda}(\rho_x+\rho^2)-(u+\mathtt{c})\rho,
\end{equation}
for the quantity
\[
\rho:=\Big(\mbox{log} (p\Psi) \Big)_x,
\] 
where \eqref{RELAZIONE} comes from the spatial part of the Lax pair \eqref{LaxPair} and \eqref{conservationlaw} comes from the time part of the Lax pair \eqref{LaxPair}.
By \eqref{conservationlaw}, for any $u(t,x)$ solution of \eqref{DP} defined on some time interval $I\subseteq\mathbb{R}$,
\begin{equation}\label{costanteneltempo}
\frac{d}{dt}\int \rho(u(t,x))\, dx=0, \quad t\in I.
\end{equation}
In \cite{Deg} $\rho$ 
is written as a formal series in powers of the spectral parameter $\lambda=\zeta^{-3}$, $\zeta\in \mathbb{R}$, 
with the coefficients determined recursively from \eqref{RELAZIONE}.
One of the possible expansions is 
 \begin{equation}\label{RHO}
 \rho=p\zeta^{-1}+\sum_{n=0}^{\infty}\rho^{(n)}\zeta^{n},
 \end{equation}
and we are interesting in studying the constants of motion
 \begin{equation}\label{seqcos}
 \Gamma^{(n)}:=\int \rho^{(n)}dx, \quad n\geq0,
 \end{equation}
 given by the series in \eqref{RHO}.
 By using \eqref{RHO} we have that \eqref{RELAZIONE} is equivalent to
 \begin{equation}\label{ordbyord}
 \begin{aligned}
 &0=\rho^{(0)}p^2+pp_x , \\
 &p-p_{xx}=\rho^{(1)}p^2+2p(\rho^{(0)})^{2}+3\del_{x}(\rho^{(0)}p) , \\
 \end{aligned}
 \end{equation}
 \begin{equation}\label{ordbyord1}
 \begin{aligned}
 \rho^{(n)}-\rho^{(n)}_{xx}&=\rho^{(n+2)}p^{2}+3\sum_{k_1+k_2=n+1}p\rho^{(k_1)}\rho^{(k_2)}
 +\sum_{k_1+k_2+k_3=n}\rho^{(k_1)}\rho^{(k_2)}\rho^{(k_3)}\\
 &+3\del_{x}(\rho^{(n+1)}p)+3\sum_{k_1+k_2=n}\rho^{(k_1)}\rho^{(k_2)}_x, \qquad n\geq0.
 \end{aligned}
 \end{equation}

From \eqref{ordbyord} we get
\begin{equation}\label{iprimi}
\rho^{(0)}=-\frac{p_x}{p}, \quad \rho^{(1)}=-\frac{p_{x}^{2}}{p^{3}}+\frac{2p_{xx}}{3p^{2}}+\frac{1}{3p}.
\end{equation}
Now we want to prove that $\rho^{(n)}(w)$ can be expressed as a power series in the variables $w$ and its derivatives in a small neighborhood of the origin of some $H^s$ Sobolev space. We refer to the Appendix to recall some definitions and facts on analytic functions on Banach spaces.

\subsection{Analyticity of composition operators on Sobolev spaces}

\begin{lem}\label{NemistkyLemma}
Let $n\geq 1$ and $f\colon \mathbb{C}^n\to \mathbb{C}$ be an analytic function on $B_{\mathbb{C}}(0, r)\times \dots \times B_{\mathbb{C}}(0, r)$ for some $r>0$. Then the composition operator 
\begin{equation}\label{Nemitsky}
T_{f}[u_1, \dots, u_n]=f( u_1, \dots, u_n)\colon B_{H^s(\T, \mathbb{C})}(0, \rho)\times \dots \times B_{H^s(\T, \mathbb{C})}(0, \rho)\to H^s, \quad \forall\,\, 0<\rho<r,\quad s>1/2
\end{equation}
 is weakly analytic on $B_{H^{s}(\T, \mathbb{C})}(0, r/2)\times \dots\times B_{H^{s}(\T, \mathbb{C})}(0, r/2)$ for $s>1/2$.
\end{lem}

\begin{proof}
First we want to show that, given $0<\rho<r$, $T_f$ maps $B_{H^s(\T, \mathbb{C})}(0, \rho)\times \dots \times B_{H^s(\T, \mathbb{C})}(0, \rho)$ into $H^s$ for $s>1/2$. Since $f$ is analytic we can write
\[
f(z_1, \dots, z_n)=\sum_{k\geq 0}\sum_{\substack{\beta\in\mathbb{N}^{\mathbb{Z}},\\ \lvert \beta \rvert=k}} f_{\beta} z^{\beta}, \quad z^{\beta}:=\prod_{i=1}^n z_i^{\beta_i}
\]
for some coefficients $f_{\beta}\in\mathbb{C}$ satisfying
\[
\sum_{k\geq 0}\sum_{\substack{\beta\in\mathbb{N}^{\mathbb{Z}},\\ \lvert \beta \rvert=k}} \lvert f_{\beta} \rvert \rho^{k}\le C, \qquad \forall 0<\rho<r
\]
for some constant $C>0$ depending only on $\rho$.
By using the algebra property of the Sobolev spaces $H^s$ with $s>1/2$, we have
\[
\lVert T_f[u_1, \dots, u_n] \rVert_{H^s(\T, \mathbb{C})}\le \sum_{k\geq 0}\sum_{\substack{\beta\in\mathbb{N}^{\mathbb{Z}},\\ \lvert \beta \rvert=k}} \lvert f_{\beta} \rvert \lVert u_1 \rVert_{H^s(\T, \mathbb{C})}^{\beta_1}\dots  \lVert u_n \rVert_{H^s(\T, \mathbb{C})}^{\beta_n}
\]
and the claim follows. In order to prove the weak analyticity of the operator $T_f$,
we have to show that for all $w_i\in B_{H^{s}(\T, \mathbb{C})}(0, r/2)$, $h_i\in H^s(\T, \mathbb{C})$, $i=1, \dots, n$ and $L\in (H^s(\T, \mathbb{C}))^*$ the function $(z_1, \dots, z_n)\mapsto L T_f(w_1+z_1 h_1, \dots, w_n+z_n h_n)$ is analytic in a neighborhood of the origin of $\mathbb{C}^n$. By Riesz theorem, for every $L\in (H^s(\T, \mathbb{C}))^*$ there exists a function $g\in H^s(\T, \mathbb{C})$ such that
\begin{equation}\label{maledetto}
\begin{aligned}
& L T_f(w_1+z_1 h_1, \dots, w_n+z_n h_n) =\sum_{m=0}^s\int \partial_x^m\Big( f \big(w_1(x)+z_1 h_1(x), \dots, w_n(x)+z_n h_n(x) \big)\Big)\,\partial_x^m g(x)\,dx\\
&=\int  f \big(w_1(x)+z_1 h_1(x), \dots, w_n(x)+z_n h_n(x) \big)\, g(x)\,dx\\
&+\sum_{m=1}^s \sum_{k=1}^m \sum_{\substack{ {\bf v}\in \mathcal{A}_{m, k},\\ {\bf p}\in \mathcal{B}_m({\bf v})}} C_{{\bf p}} \int \big( D^{ {\bf v}} f\big)  \prod_{i=1}^n \prod_{j=1}^{\lvert {\bf v}_i \rvert} (\partial_x^{{\bf p}^{(i)}_{j}}w_{i}+z_{i} \partial_x^{{\bf p}^{(i)}_{j}}   h_{i} )\,\partial_x^m g(x)\,dx
\end{aligned}
\end{equation}
where $C_{{\bf p}}$ are combinatorial factors,
\[
\mathcal{A}_{m, k}:=\{ {\bf v}\in \{0, \dots, m\}^n\, :\,\lvert {\bf v} \rvert:=\sum_{i=1}^n {\bf v}_i=k \}, 
\]
\[
\mathcal{B}_r({\bf v}):=\{ {\bf p}=({\bf p}^{(1)}, \dots, {\bf p}^{(n)}),\, {\bf p}^{(i)}\in\{1, \dots, m \}^{\lvert {\bf v}\rvert},\,\sum_{i=1}^n {\bf p}^{(i)}=r  \} \qquad \mbox{for}\,\,\, {\bf v}\in \mathcal{A}_{m, k}
\]
and
\[
D^{{\bf v}} f:= \partial_{   {\bf v}_1 \dots {\bf v}_n  }   f=\frac{\partial^k}{\partial_{z_1}^{{\bf v}_1}\dots \partial_{z_n}^{{\bf v}_n}   }, \qquad  \partial_{i} :=\partial_{z_i}.
\]
The last summand in \eqref{maledetto} can be written as
\begin{equation}\label{maledetto2}
\begin{aligned}
&\sum_{m=1}^{s-1} \sum_{k=1}^m  \sum_{\substack{ {\bf v}\in \mathcal{A}_{m, k},\\ {\bf p}\in \mathcal{B}_m({\bf v})}} C_{{\bf p}} \int \big( D^{ {\bf v} } f\big)   \prod_{i=1}^n \prod_{j=1}^{\lvert {\bf v}_i \rvert} (\partial_x^{{\bf p}^{(i)}_{j}}w_{i}+z_{i} \partial_x^{{\bf p}^{(i)}_{j}}   h_{i} )\,\partial_x^m g(x)\,dx\\
&+ \sum_{k=2}^s  \sum_{\substack{ {\bf v}\in \mathcal{A}_{m, k},\\ {\bf p}\in \mathcal{B}_s({\bf v})}} C_{{\bf p}} \int \big( D^{ {\bf v}} f\big)  \prod_{i=1}^n \prod_{j=1}^{\lvert {\bf v}_i \rvert} (\partial_x^{{\bf p}^{(i)}_{j}}w_{i}+z_{i} \partial_x^{{\bf p}^{(i)}_{j}}   h_{i} )\,\partial_x^s g(x)\,dx\\
&+ \sum_{i=1}^n \int (\partial_{i} f)\,(\partial_x^{s}w_{i}+z_i \partial_x^{s}h_{i} )\,\partial_x^s g\,dx.
\end{aligned}
\end{equation}
 We want to prove that there exist the derivatives in the complex variable $z_i$ of the function $L T_f(w_1+z_1 h_1, \dots, w_n+z_n h_n)$: to do that it is sufficient to prove that the derivative in $z_i$ of the integrands in \eqref{maledetto} is $L^1$ uniformly in the parameter $z$, since by dominated convergence we can pass the derivative inside the integral and use the analyticity of $f$. Hence we now show that the following sum
 \begin{equation}\label{lucarelli}
\begin{aligned} 
&  \int \lvert (\partial_{z_{\xi}} f)h_{\xi}\, g(x)\, \rvert dx+\sum_{m=1}^{s-1} \sum_{k=1}^m  \sum_{\substack{ {\bf v}\in \mathcal{A}_{m, k},\\ {\bf p}\in \mathcal{B}_m({\bf v})}} C_{{\bf p}} \int \lvert \big( D^{ {\bf v}+\mathtt{e}_{\xi} } f\big)\,h_{\xi}\,   \prod_{i=1}^n \prod_{j=1}^{\lvert {\bf v}_i \rvert} (\partial_x^{{\bf p}^{(i)}_{j}}w_{i}+z_{i} \partial_x^{{\bf p}^{(i)}_{j}}   h_{i} )\,\partial_x^m g(x)\,\rvert\,dx\\
&+\sum_{m=1}^{s-1} \sum_{k=1}^m  \sum_{\substack{ {\bf v}\in \mathcal{A}_{m, k},\\ {\bf p}\in \mathcal{B}_m({\bf v})}} C_{{\bf p}} \int \lvert \big( D^{ {\bf v} } f\big) \prod_{i=1, i\neq {\xi}}^n \prod_{j=1}^{\lvert {\bf v}_i \rvert} (\partial_x^{{\bf p}^{(i)}_{j}}w_{i}+z_{i} \partial_x^{{\bf p}^{(i)}_{j}}   h_{i} )\\
&\qquad\qquad\times\Big(\sum_{j=1}^{\lvert {\bf v}_{\xi} \rvert}  \prod_{b\neq j} (\partial_x^{{\bf p}^{({\xi})}_{b}}w_{{\xi}}+z_{{\xi}} \partial_x^{{\bf p}^{({\xi})}_{b}}   h_{{\xi}})   \partial_x^{{\bf p}_j^{({\xi})}} h_{\xi}   \Big)\,\partial_x^m g(x)\,\rvert\,dx\\
&+ \sum_{k=2}^s  \sum_{\substack{ {\bf v}\in \mathcal{A}_{m, k},\\ {\bf p}\in \mathcal{B}_s({\bf v})}} C_{{\bf p}} \int \lvert \big( D^{ {\bf v}+\mathtt{e}_{\xi}} f\big)\, h_{\xi}\,  \prod_{i=1}^n \prod_{j=1}^{\lvert {\bf v}_i \rvert} (\partial_x^{{\bf p}^{(i)}_{j}}w_{i}+z_{i} \partial_x^{{\bf p}^{(i)}_{j}}   h_{i} )\,\partial_x^s g(x)\,\rvert\,dx\\  
&+ \sum_{k=2}^s  \sum_{\substack{ {\bf v}\in \mathcal{A}_{m, k},\\ {\bf p}\in \mathcal{B}_s({\bf v})}} C_{{\bf p}} \int \lvert \big( D^{ {\bf v}} f\big)\,   \prod_{i=1}^n \prod_{j=1}^{\lvert {\bf v}_i \rvert} (\partial_x^{{\bf p}^{(i)}_{j}}w_{i}+z_{i} \partial_x^{{\bf p}^{(i)}_{j}}   h_{i} ) \\
&\qquad \qquad\times\Big(\sum_{j=1}^{\lvert {\bf v}_{\xi} \rvert}  \prod_{b\neq j} (\partial_x^{{\bf p}^{({\xi})}_{b}}w_{{\xi}}+z_{{\xi}} \partial_x^{{\bf p}^{({\xi})}_{b}}   h_{{\xi}})   \partial_x^{{\bf p}_j^{({\xi})}} h_{\xi}   \Big)
\,\partial_x^s g(x)\,\rvert\,dx\\
&+ \sum_{i=1}^n \int \lvert (\partial_{i\,{\xi}} f)\,h_{\xi}\,(\partial_x^{s}w_{i}+z_i \partial_x^{s}h_{i} )\,\partial_x^s g\,\rvert\,dx+  \int \lvert (\partial_{{\xi}} f)\partial_x^{s}h_{{\xi}}\,\partial_x^s g\, \rvert\,dx
 \end{aligned}
 \end{equation}
 is finite for some ${\xi}\in\{ 1,\dots, n \}$.
Fix $\lvert z_i \rvert<\min\{r/2\lVert h_i \rVert_{H^{s}(\T, \mathbb{C})}, r/2\}$ for $i=1, \dots, n$. First we bound the derivatives of $f$
\[
\lvert \big( D^{{\bf v}} f\big)\big(w_1+z_1 h_1, \dots, w_n+z_n h_n \big)\rvert_{L^{\infty}}<\infty \qquad \forall w_i\in B_{H^s(\T, \mathbb{C})}(0, r/2), \,\,h_i\in H^s(\T, \mathbb{C})
\]
since the derivatives of $f$ are analytic on $B_{\mathbb{C}}(0, r)\times \dots \times B_{\mathbb{C}}(0, r)$ and
\[
\lvert  w_i+z_i h_i  \rvert_{L^{\infty}}<r, \qquad i=1, \dots, n.
\]
Since $w_i$ and $h_i$ belong to $H^s(\T, \mathbb{C})$ then, for $1\le k \le m$, $1\le m \le s$, ${\bf v}\in \mathcal{A}_{m, k}$, ${\bf p}\in\mathcal{B}_m({\bf v})$ we have by Cauchy-Schwarz and Sobolev embeddings
\begin{align*}
&\int \lvert   \prod_{i=1}^n \prod_{j=1}^{\lvert {\bf v}_i \rvert} (\partial_x^{{\bf p}^{(i)}_{j}}w_{i}+z_{i} \partial_x^{{\bf p}^{(i)}_{j}}   h_{i} )\,h_{\xi}\,\partial_x^m g(x) 	\rvert\,dx\le \prod_{i=1}^n \prod_{j=1}^{\lvert {\bf v}_i \rvert}\lvert   (\partial_x^{{\bf p}^{(i)}_{j}}w_{i}+z_{i} \partial_x^{{\bf p}^{(i)}_{j}}   h_{i} ) \rvert_{L^{\infty}(\T, \mathbb{C})} \lVert h_{\xi} \rVert_{L^2(\T, \mathbb{C})} \lVert \partial_x^m g \rVert_{L^2(\T, \mathbb{C})}\\
&\le \prod_{i=1}^n \prod_{j=1}^{\lvert {\bf v}_i \rvert} (\lVert \partial_x^{{\bf p}^{(i)}_{j}}w_{i} \rVert_{H^1(\T, \mathbb{C})}+\lvert z_{i} \rvert \lVert \partial_x^{{\bf p}^{(i)}_{j}}   h_{i}\rVert_{H^1(\T, \mathbb{C})} ) \lVert h_{\xi} \rVert_{L^2(\T, \mathbb{C})} \lVert \partial_x^m g \rVert_{L^2(\T, \mathbb{C})}\le  r^m \, \lVert h_{\xi} \rVert_{L^2(\T, \mathbb{C})} \lVert  g \rVert_{H^s(\T, \mathbb{C})}
\end{align*}
and similarly
\begin{align*}
\int \lvert \prod_{i=1, i\neq \xi}^n \prod_{j=1}^{\lvert {\bf v}_i \rvert} (\partial_x^{{\bf p}^{(i)}_{j}}w_{i}+z_{i} \partial_x^{{\bf p}^{(i)}_{j}}   h_{i} )\,\Big(\sum_{j=1}^{\lvert {\bf v}_{\xi} \rvert}  \prod_{b\neq j} (\partial_x^{{\bf p}^{(\xi)}_{b}}w_{\xi}+z_{\xi} \partial_x^{{\bf p}^{(\xi)}_{b}}   h_{\xi})  & \partial_x^{{\bf p}_j^{(\xi)}} h_{\xi}   \Big)\,\partial_x^m g(x) \rvert\,dx\\
&\le m r^{m-1}\, \lVert h_{\xi} \rVert_{L^2(\T, \mathbb{C})} \lVert  g \rVert_{H^s(\T, \mathbb{C})}.
\end{align*}

We bounded the first three terms in \eqref{lucarelli}. The fourth and the fifth terms in \eqref{lucarelli} have similar bounds and the proof follows the arguments above; we remark only that the Cauchy-Schwarz inequality has to be applied to the $L^2$-product of $\partial_x^s g$ with $\partial_x^{{\bf p}_j^{(\xi)}} h_{\xi} $.
For the last two summands of \eqref{lucarelli} we have
\begin{align*}
\int \lvert h_r (\partial_x^{s}w_{i}+z_i \partial_x^{s}h_{i} )\,\partial_x^s g \rvert\,dx &\le \lvert h_r \rvert_{L^{\infty}(\T, \mathbb{C})} \lVert \partial_x^{s}w_{i}+z_i \partial_x^{s}h_{i} \rVert_{L^2(\T, \mathbb{C})}\lVert \partial_x^s g  \rVert_{L^2(\T, \mathbb{C})}\\
&\le  \lvert h_r \rvert_{L^{\infty}(\T, \mathbb{C})}  (\lVert w_{i} \rVert_{H^s(\T, \mathbb{C})}+\lvert z_i \rvert\lVert h_{i} \rVert_{H^s(\T, \mathbb{C})})\lVert  g  \rVert_{H^s(\T, \mathbb{C})} \le r\, \lVert h_r \rVert_{H^1(\T, \mathbb{C})} \lVert  g  \rVert_{H^s(\T, \mathbb{C})}
\end{align*}
and 
\[
\int \lvert \partial_x^{s}h_{\xi} \,\partial_x^s g \rvert\,dx\le \lVert h_{\xi}\rVert_{H^s(\T, \mathbb{C})} \lVert g \rVert_{H^s(\T, \mathbb{C})}.
\]

%
\end{proof}

\begin{lem}\label{secondlem}
Let $\sigma\in\mathbb{N}$ and $f\colon \mathbb{C}\to \mathbb{C}$ be analytic on a ball $B_{\mathbb{C}}(0, r)$. Then there exists a function $g\colon \mathbb{C}^{\sigma+1}\to \mathbb{C}$ analytic on $B_{\mathbb{C}}(0, r)\times \dots\times B_{\mathbb{C}}(0, r)$ such that $\big(\partial_x^{\sigma}\circ T_f \big)(w)$ is the restriction to $$w_0=w, \dots, w_{\sigma}=\partial_x^{\sigma} w$$ of the composition operator $T_g[w_0, \dots, w_{\sigma}]\colon H^s(\T, \mathbb{C})\times\dots \times H^s(\T, \mathbb{C}) \to H^s(\T, \mathbb{C})$ for $s>1/2$. Moreover $T_g$
is analytic on $B_{H^s(\T, \mathbb{C})}(0, r/2)\times \dots \times B_{H^{s}(\T, \mathbb{C})}(0, r/2)$ for $s>1/2$. 
\end{lem}
\begin{proof}
By the chain rule
\begin{equation}\label{faadibruno}
\partial_x^{\sigma} f(w)=\sum_{k=1}^{\sigma} \sum_{p_1+\dots+p_k=\sigma} C_k f^{(k)}(w)(\partial_x^{p_1} w) \dots (\partial_x^{p_k} w), 
\end{equation}
hence the function $\partial_x^{\sigma}\circ T_f$ is the restriction of the composition operator $T_g$ on $w_0=w, \dots, w_{\sigma}=\partial_x^{\sigma} w$ for a function $g=g(z_0, \dots, z_{\sigma})\colon \mathbb{C}^{\sigma+1}\to \mathbb{C}$ which has the form
\[
\sum_{k=1}^{\sigma} \sum_{p_1+\dots+p_k=\sigma} C_k f^{(k)}(z_0)z_{p_1} \dots z_{p_k}
=\sum_{k=1}^{\sigma} \tilde{C}_k f^{(k)}(z_0) z_1^{\alpha_1^{(k)}}\dots z_{\sigma}^{\alpha_{\sigma}^{(k)}}
\]
for some $\alpha_i^{(k)}\in \mathbb{N}$ and some positive constants $\tilde{C}_k$.\\
The function $g$ is clearly analytic on $B_{\mathbb{C}}(0, r)\times \dots\times B_{\mathbb{C}}(0, r)$ and we have the weakly analyticity of $T_g$ by using Lemma \ref{NemistkyLemma}. The fact that $T_g$ is locally bounded as operator from $H^s(\T, \mathbb{C})\times\dots \times H^s(\T, \mathbb{C})$ to $H^s(\T, \mathbb{C})$ follows trivially by the following estimate, obtained by exploiting the algebra property of $H^s(\T, \mathbb{C})$ with $s>1/2$ and the analyticity of $f$,
\[
\lVert \partial_x^{\sigma} f(w) \rVert_{H^s(\T, \mathbb{C})}\le \sum_{k=1}^{\sigma} \sum_{p_1+\dots+p_k=\sigma} C_k \lVert f^{(k)}(w_0) \rVert_{H^s(\T, \mathbb{C})} \lVert w_{p_1}\rVert_{H^s(\T, \mathbb{C})} \dots \lVert w_{p_k}\rVert_{H^s(\T, \mathbb{C})}.
\]
\end{proof}


The function $p(y)=-(\mathtt{c}+y)^{1/3}$ is analytic in $\{ y\in \mathbb{C} : \lvert y \rvert< \lvert \mathtt{c} \rvert \}$, hence by Lemma \ref{NemistkyLemma} the map $p(w)=T_p[w]=-(\mathtt{c}+w)^{1/3}$ defined in \eqref{emme} is weakly analytic in $B_{H^s(\T, \mathbb{C})}(0,  \lvert \mathtt{c} \rvert /2)$. Moreover $T_p$ is locally bounded, hence $p(w)$ is analytic in $B_{H^s(\T, \mathbb{C})}(0,  \lvert \mathtt{c} \rvert /2)$ and it can be represented by its Taylor expansion at the origin
\begin{equation}\label{taylor}
p(w)=\sum_{n\geq 0} \frac{p^{(n)}(0)}{n!}\,w^n.
\end{equation}

\begin{remark}\label{realta}
We note that the function $p(y)=-(\mathtt{c}+y)^{1/3}$ is real on real, namely it assumes real valued when it is restricted to the real line. Then its restriction to $\mathbb{R}$ is a real analytic function.\\
As a consequence, it is easy to see that the composition operator $T_p$ is real on real and then it is analytic on $H^s:=H^s(X, \mathbb{R})$, $X=\T, \mathbb{R}$.
\end{remark}

\subsubsection{Class of differential polynomials}\label{diffpol}
We introduce a class of differential polynomials to which the Taylor expansion of the $\rho^{(n)}$ belongs. The particular form of these polynomial results to be fundamental for the Sobolev estimates on the constants of motion which we construct.

\medskip

We define
\begin{equation}\label{definterval}
\mathcal{J}_n^q:=\{ \alpha\in \mathbb{N}^{\{0, \dots, n\}} : \sum_{i=0}^n \alpha_i=q, \,\, \sum_{i=0}^n i \alpha_i \le n \}
\end{equation}
and for $\alpha\in \mathcal{J}_n^q$, $w=(w_0, \dots, w_n)$, $w_i:=\partial_x^i w$ the monomial
\begin{equation}\label{defmonomial}
w^{\alpha}=\prod_{i=0}^n w_i^{\alpha_i}=\prod_{i=0}^n (\partial_x^i w)^{\alpha_i}.
\end{equation}
We denote by $\calP_n^q$ the class of formal homogenous polynomials of degree $q$ and order $n$ of the form
\[
f=\sum_{\alpha\in \mathcal{J}_n^q} f_{\alpha} w^{\alpha}, \quad f_{\alpha}\in \mathbb{C}.
\]
We denote by $\calP_n^{\le q}$ the class of formal polynomials of degree at most $q$ and order $n$ of the form
\[
f=\sum_{k=0}^q f_k, \quad f_k\in\mathcal{P}_n^k.
\]
We denote by $\Sigma_n^q$ the class of formal power series of degree at least $q$ of the form
\[
f=\sum_{k=q}^{\infty}f_{n, k}, \quad f_{n, k}\in \calP_n^k.
\]
The Taylor series \eqref{taylor} is an element of $\Sigma_0^0$.

%

%
\begin{lem}\label{polem}
Let $\sigma, n, m, q, r\in\mathbb{N}$. Then
\begin{enumerate}
\item If $f\in \calP_n^q$, $g\in\calP_m^r$ then 
\[
f+g\in \calP_{\max\{n, m \}}^{\le \max\{ q, r \}}, \quad f\,g\in\calP_{\max\{n, m \}}^{\le q+r}.
\]
\item  The operator $\partial_x^{\sigma}$ maps $\calP_n^q$ into $\calP_{n+\sigma}^q$.
\end{enumerate}
\end{lem}
\begin{proof}
\textit{Proof of $(1)$}: for the sum the proof is trivial. For the product, suppose that $m\geq n$, the claim follows by the fact that
\[
w_0^{\alpha_0}\dots w_n^{\alpha_n}\, w_0^{\beta_0}\dots w_m^{\beta_m}=\prod_{i=0}^{n} w_i^{\alpha_i+\beta_i}\,w_{n+1}^{\beta_{n+1}}\dots w_m^{\beta_m}
\]
where $\lvert\alpha \rvert=q$ and $\lvert \beta \rvert=r$.

\vspace{0.8em}

\textit{Proof of $(2)$}:
clearly it is sufficient to look at the action of $\partial_x^{\sigma}$ on the monomials.
Fixed $i\in\mathbb{N}$, we have that
\[
\partial_x^p w_i=w_{i+p} \quad \mbox{for}\,\,p\in\mathbb{N}
\]
and by the chain rule
\[
\partial_x^{j} w_i^{\alpha_i}=\sum_{k=1}^j \sum_{p_1+\dots+p_k=j} C_k w_i^{\alpha_i-k} (\partial_x^{p_1} w_i)\dots (\partial_x^{p_k} w_i)=\sum_{k=1}^j \sum_{p_1+\dots+p_k=j} C_k w_i^{\alpha_i-k} w_{i+p_1}\dots  w_{i+p_k}
\]
is a function of variables $w_i, \dots, w_{i+j}$. It is easy to see that, in these variables, the degree of homogeneity has not been changed, namely it is already $\alpha_i$. Hence
\[
\partial_x^{\sigma} w^{\alpha}=\sum_{j_0+\dots+j_n=\sigma } C_{j_0 \dots j_n} (\partial_x^{j_0} w_0^{\alpha_0})\dots (\partial_x^{j_n} w_n^{\alpha_n})
\]
is a function of the variables $w_0, \dots, w_{n+\sigma}$ with homogeneity degree $\alpha_0+\dots+\alpha_n= q$.
\end{proof}
%
The following remark is fundamental for getting bounds on the Sobolev norms of the constants of motion.
\begin{remark}\label{remarkaffine}
Let $f\in \Sigma_n^q$ for some $n, q\geq 0$. We note that $f$ is necessarily affine in the variable $w_n=\partial_x^n w$, namely in the highest order derivative. Indeed, $\mathcal{M}(\alpha)=n=\sum_{i=0}^n i \alpha_i$ for $\alpha\in \mathcal{J}_n^q$, hence
\[
\alpha_0=q-1, \quad \alpha_i =0, \quad \alpha_n=1, \quad i=1, \dots, n-1, \qquad \mbox{or} \qquad \alpha_n=0.
\]
 So
\[
f=w_n\,\sum_{k\geq q} f_{(k-1, 0, \dots, 0, 1)} w^k+\sum_{k\geq q} \sum_{\substack{\alpha\in\mathcal{J}_n^q,\\ \alpha_n=0}} f_{\alpha} w^{\alpha}.
\]
\end{remark}

\subsection{The structure of the conserved quantities $\rho^{(n)}$}

The functions $\rho^{(n)}$ are given by sums and products of $p$ and its $x$-derivatives. We want to show that the $\rho^{(n)}$ are composition operators for analytic and real on real functions of $w$ and its derivatives and that the Taylor expansions of these operators belong to some $\Sigma_n^q$ (recall the definitions given in Section \ref{diffpol}). This fact will allow to prove the bound \eqref{stimafondam} and consequently the estimates in Theorem \ref{costantiMotoDP}-$(iii)$.

\begin{remark}
Given two composition operators $T_f$ and $T_g$ we have that $T_f+T_g=T_{f+g}$ and $T_{f}T_g=T_{f g}$. Hence if $f, g$ are analytic we can apply to $T_f+T_g=T_{f+g}$ and $T_{f}T_g=T_{f g}$ Lemmata \ref{NemistkyLemma}, \ref{secondlem}.
\end{remark}

\begin{lem}\label{lemmaron}
Fix $n\in\mathbb{N}$. Then there exists a function $f_n\colon \mathbb{C}^{n+2}\to \mathbb{C}$ real on real, analytic on $B_{\mathbb{C}}(0,  \lvert \mathtt{c} \rvert)\times \dots\times B_{\mathbb{C}}(0,   \lvert \mathtt{c} \rvert)$ such that
$$\rho^{(n)}(w)=T_{f_n}[w, w_x, \dots, \partial_{x}^{n+1} w].$$
Moreover the Taylor series of $T_{f_n}$ restricted to $w_0=w, \dots, w_{n+1}=\partial_x^{n+1} w$ belongs to $\Sigma_{n+1}^{0}$. 
\end{lem}

\begin{proof}
Let us start from $\rho^{(0)}$ and $\rho^{(1)}$, and then we argue by induction on $n$.\\
Recalling that $p$ is analytic as function of the variable $w$, by Lemma \ref{secondlem} $p_x$ is analytic as function of the two variables $w$, $w_x$. By \eqref{iprimi}, we have that
$
\rho^{(0)} = \frac{1}{3} (-\mathtt{c}-w)^{-1} w_x,
$
and since the function $f_0:\mathbb{C}^2 \to \mathbb{C}$ given by 
$$f_0(z_0,z_1):= \frac{1}{3} (-\mathtt{c}-z_0)^{-1} z_1$$
is real on real and analytic in $B_{H^s}(0,|c|) \times B_{H^s}(0,|c|)$, by Lemma \ref{NemistkyLemma} and by local boundedness of the operator $T_{f_0}$ we get that $\rho^{(0)}$ is analytic in $B_{H^s}(0,|\mathtt{c}|/2) \times B_{H^s}(0,|\mathtt{c}|/2)$. From the explicit formula of $f_0$ one can deduce that the Taylor series of $T_{f_0}$ restricted to $w_0=w, w_1=w_x$ belongs to $\Sigma_{1}^{0}$.\\
 Similarly, we obtain that $\rho^{(1)}$ is real on real and analytic in the variables $w$, $w_x$ and $w_{xx}$, since it can be written as a composition operator for the following analytic function
\begin{align*}
f_1(z_0,z_1,z_2) &:= - \frac{1}{9} (-\mathtt{c}-z_0)^{-7/3} z_1^2 + \frac{2}{3} \left( - \frac{2}{9} \frac{ z_1^2  }{ (-\mathtt{c}-z_0)^{5/3} }-\frac{1}{3}(-\mathtt{c}-z_0)^{-2/3} z_2 \right) + \frac{1}{3} (-\mathtt{c}-z_0)^{-1/3}. 
\end{align*}
Furthermore, from the explicit formula of $f_1$ one can deduce that the Taylor series of $T_{f_1}$ restricted to $w_0=w, \dots, w_2=w_{xx}$ belongs to $\Sigma_{2}^{0}$.

Now we assume that the thesis holds for $\rho^{(k)}$, $k \leq n+1$, and we only have to control that $\rho^{(n+2)}$ depends only on $w$, $w_x$, $\ldots$, $\del_x^{n+3}w$. But by recalling \eqref{ordbyord1}, we just observe that
\begin{itemize}
\item $\rho^{(n)}(w) = T_{f_n}[w,\ldots,\del_x^{n+1}w]$ for some $f_n:\mathbb{C}^{n+2} \to \mathbb{C}$ analytic on $\times_{i=1}^{n+2} B_{\mathbb{C}}(0,|\mathtt{c}|)$, by inductive hypothesis; 
\item $\rho_{xx}^{(n)}(w) = T_{g_n}[w,\ldots,\del_x^{n+3}w]$ for some $g_n:\mathbb{C}^{n+4} \to \mathbb{C}$ analytic on $\times_{i=1}^{n+4} B_{\mathbb{C}}(0,|\mathtt{c}|)$, by Lemma \ref{secondlem} and by inductive hypothesis; 
\item $p(w) \rho^{(k_1)}(w) \rho^{(k_2)}(w) = T_{h_{k_1,k_2}}[w,\ldots,\del_x^{\max(k_1,k_2)+1}w]$ (where $k_1+k_2=n+1$), for some $h_{k_1,k_2}:\mathbb{C}^{\max(k_1,k_2)+2} \to \mathbb{C}$ analytic on $\times_{i=1}^{\max(k_1,k_2)+2} B_{\mathbb{C}}(0,|\mathtt{c}|)$, by inductive hypothesis and by the above remark; 
\item $\rho^{(k_1)}(w) \rho^{(k_2)}(w) \rho^{(k_3)}(w)= T_{l_{k_1,k_2,k_3}}[w,\ldots,\del_x^{\max(k_1,k_2,k_3)+1}w]$ ($k_1+k_2+k_3=n$), for some $l_{k_1,k_2,k_3}:\mathbb{C}^{\max(k_1,k_2,k_3)+2} \to \mathbb{C}$ analytic on $\times_{i=1}^{\max(k_1,k_2,k_3)+2} B_{\mathbb{C}}(0,|\mathtt{c}|)$, by inductive hypothesis and by the above remark; 
\item $\del_x(\rho^{(n+1)}p)(w)=(\del_x\rho^{(n+1)})(w) p(w) + \rho^{(n+1)}(w) p_x(w) = T_{m_{n+1}}[w,\ldots,\del_x^{n+3}w]$ for some $m_{n+1}:\mathbb{C}^{n+4} \to \mathbb{C}$ analytic on $\times_{i=1}^{n+4} B_{\mathbb{C}}(0,|\mathtt{c}|)$, by inductive hypothesis and by the above remark;
\item $\rho^{(k_1)}(w) \rho_x^{(k_2)}(w) = T_{v_{k_1,k_2}}[w,\ldots,\del_x^{\max(k_1,k_2+1)+1}w]$ ($k_1+k_2=n$), for some $v_{k_1,k_2}:\mathbb{C}^{\max(k_1,k_2+1)+2} \to \mathbb{C}$ analytic on $\times_{i=1}^{\max(k_1,k_2+1)+2} B_{\mathbb{C}}(0,|\mathtt{c}|)$, by inductive hypothesis and by the above remark.
\end{itemize}
Furthermore, again by using formula \eqref{ordbyord1}, we have that
\begin{itemize}
\item the Taylor series of $T_{f_n}$ restricted to $w_0=w, \dots, w_{n+1}=\partial_x^{n+1} w$ belongs to $\Sigma_{n+1}^{0}$, by inductive hypothesis;
\item the Taylor series of $T_{g_n}$ restricted to $w_0=w, \dots, w_{n+3}=\partial_x^{n+3} w$ belongs to $\Sigma_{n+3}^{0}$, by inductive hypothesis and by Lemma \ref{polem};
\item the Taylor series of $T_{ h_{k_1,k_2} }$ ($k_1+k_2=n+1$) restricted to $w_0=w, \dots, w_{\max(k_1,k_2)+1}=\partial_x^{\max(k_1,k_2)+1} w$ belongs to $\Sigma_{n+2}^0$, by inductive hypothesis and by Lemma \ref{polem};
\item the Taylor series of $T_{ l_{k_1,k_2,k_3} }$ ($k_1+k_2+k_3=n$) when restricted to $w_0=w, \dots, w_{\max(k_1,k_2,k_3)+1}=\partial_x^{\max(k_1,k_2,k_3)+1} w$ belongs to $\Sigma_{n+1}^0$, by inductive hypothesis and by Lemma \ref{polem};
\item the Taylor series of $T_{ m_{n+1} }$ restricted to $w_0=w, \dots, w_{n+3}=\partial_x^{n+3} w$ belongs to $\Sigma_{n+3}^0$, by inductive hypothesis and by Lemma \ref{polem};
\item the Taylor series of $T_{ v_{k_1,k_2} }$ ($k_1+k_2=n$) restricted to $w_0=w, \dots, w_{\max(k_1,k_2+1)+1}=\partial_x^{\max(k_1,k_2+1)+1} w$ belongs to $\Sigma_{n+2}^0$, by inductive hypothesis and by Lemma \ref{polem}.
\end{itemize}
This implies that the Taylor series of $T_{f_{n+2}}$ restricted to $w_0=w, \dots,w_{n+3}=\del_x^{n+3}w$ belongs to $\Sigma_{n+3}^0$.
\end{proof}
By Remark \ref{realta} the composition operators $\rho^{(n)}$ are real analytic and 
by Theorem \ref{TruboTeo} $\rho^{(n)}(w)$ can be represented by their Taylor expansion at the origin as functions of $w_0:=w, \dots, w_n:=\partial_x^n w$ if $w$ belongs to a sufficiently small ball of $H^{s+n}$ centered at the origin. For instance we can write
\begin{align}
p &= -\mathtt{c}^{1/3} - \frac{1}{ 3 \mathtt{c}^{2/3} } w + \frac{1}{ 9 \mathtt{c}^{5/3} } w^{2}+g_0(w), \label{p} \\
\rho^{(0)}&=-\frac{w_x}{3\mathtt{c}}+\frac{ww_x}{3\mathtt{c}^2}+g_1(w,w_x), \label{rho0} \\
\rho^{(1)}&=-\frac{1}{3\mathtt{c}^{1/3}}+\frac{1}{9\mathtt{c}^{4/3}} w-\frac{2}{9\mathtt{c}^{4/3}} w_{xx} -\frac{2}{27\mathtt{c}^{7/3}} w^{2} +\frac{8}{27\mathtt{c}^{7/3}} ww_{xx} +\frac{7}{27\mathtt{c}^{7/3}} w_{x}^{2}+g_{2}(w,w_x,w_{xx}), \label{rho1}
\end{align}
where $g_0$, $g_1$ and $g_2$ have a zero of order $3$ at the origin.

We remark that $\Gamma^{(n)}$ defined in \eqref{seqcos} is the integral (on the torus $\T$ or on $\mathbb{R}$) of elements of $\Sigma_n^q$. In the following lemma we prove an estimate on Sobolev spaces for this kind of functions.

\begin{prop}\label{LemmaStimeTaylor}

Fix $n\in\mathbb{N}$ and let $F(w):=\int f(w, \dots, \partial_x^n w)\,dx$, where $f\colon \mathbb{C}^{n+1}\to\mathbb{C}$ is real on real, analytic on $B_{\mathbb{C}}(0, r)\times \dots \times B_{\mathbb{C}}(0, r)$ for some $r>0$ and the Taylor expansion of $f(w, \dots, \partial_x^n w)$ at the origin belongs to $\Sigma_n^q$. Then $F\colon H^n\to\mathbb{C}$ is analytic on $B_{H^n}(0, r/2)$ and the following estimate holds
\begin{equation}\label{stimafondam}
\lvert F(w) \rvert\le C(n, r)\lVert w \rVert^q_{H^n}
\end{equation}
\end{prop}
\begin{proof}
First we prove the bound \eqref{stimafondam}, which implies also that $F$ is locally bounded on $B_{H^n}(0, r)$. We note that, since the Taylor series of $f(w, \dots, \partial_x^n w)$ belongs to $\Sigma_n^q$, we can write by Remark \ref{remarkaffine}
\[
F(w)=\int \sum_{k\geq q }\sum_{\alpha\in\mathcal{J}_n^k} f_{\alpha} w^{\alpha}\,dx=\int_{\T }w_n\,\sum_{k\geq q} f_{(k-1, 0, \dots, 0, 1)} w^k\,dx+\int\sum_{k\geq q} \sum_{\substack{\alpha\in\mathcal{J}_n^k,\\ \alpha_n=0}} f_{\alpha} w^{\alpha}\,dx.
\]
Hence by Cauchy-Schwarz and Sobolev embeddings
\begin{equation}\label{briscola}
\begin{aligned}
\lvert F(w) \rvert &\le\sum_{k\geq q} \lvert f_{(k-1, 0, \dots, 0, 1)} \rvert  \int \lvert \partial_x^n w \rvert \lvert w^{k-1} \rvert\,dx+\sum_{k\geq q} \sum_{\substack{\alpha\in\mathcal{J}_n^k,\\ \alpha_n=0}}\lvert f_{\alpha} \rvert \int \lvert w^{\alpha} \rvert\,dx\\
&\le \sum_{k\geq q} \lvert f_{(k-1, 0, \dots, 0, 1)} \rvert  \lvert w \rvert_{L^{\infty}}^{k-2}  \int \lvert \partial_x^n w  \rvert\lvert w \rvert\,dx+\sum_{k\geq q} \sum_{\substack{\alpha\in\mathcal{J}_n^k,\\ \alpha_n=0}}\lvert f_{\alpha} \rvert \prod_{i=0}^{k-2} \lVert w \rVert_{H^{i+1}} \lVert w \rVert_{H^{n-1}}^2\\
&\le  \sum_{k\geq q} \lvert f_{(k-1, 0, \dots, 0, 1)} \rvert  \lvert w \rvert_{L^{\infty}}^{k-2}  \lVert w \rVert_{H^n}\lVert w \rVert_{L^2}+\sum_{k\geq q} \sum_{\substack{\alpha\in\mathcal{J}_n^k,\\ \alpha_n=0}}\lvert f_{\alpha} \rvert \lVert w \rVert_{H^{n}}^{k}\\
&\le  \sum_{k\geq q} \lvert f_{(k-1, 0, \dots, 0, 1)} \rvert  \lVert w \rVert_{H^1}^{k-1}  \lVert w \rVert_{H^n}+\sum_{k\geq q} \sum_{\substack{\alpha\in\mathcal{J}_n^k,\\ \alpha_n=0}}\lvert f_{\alpha} \rvert \lVert w \rVert_{H^{n}}^{k}\\ 
&\le \lVert w \rVert_{H^n}^q \Big( \sum_{k\geq q} \lvert f_{(k-1, 0, \dots, 0, 1)} \rvert   \lVert w \rVert^{k-q}_{H^n}+\sum_{k\geq q} \sum_{\substack{\alpha\in\mathcal{J}_n^k,\\ \alpha_n=0}}\lvert f_{\alpha} \rvert \lVert w \rVert_{H^{n}}^{k-q}\Big)\\
&\le  \lVert w \rVert_{H^n}^q \sum_{k\geq q} \sum_{\substack{\alpha\in\mathcal{J}_n^k}}\lvert f_{\alpha} \rvert r^{k-q}\le r^{-q} \lVert w \rVert_{H^n}^q \sum_{k\geq q} \sum_{\substack{\alpha\in\mathcal{J}_n^k}}\lvert f_{\alpha} \rvert r^{k}\le \frac{C(f, n, r)}{r^q}  \lVert w \rVert_{H^n}^q 
\end{aligned}
\end{equation}

Now we prove the weakly analyticity of $F$. Since $\mathbb{C}^*=\mathbb{C}$ it is sufficient to show that for every $w\in B_{H^s}(0, r/2)$ and $h\in H^s$ there exists $a=a(h)>0$ such that the function $z\mapsto F(w+z h)$ is analytic as function of a complex variable in $B_{\mathbb{C}}(0, a)$. \\
If $\lvert z \rvert<\min\{  \frac{r}{2 \lVert h \rVert_{H^n}} , \frac{r}{2} \}$ then
\begin{equation}\label{sharp}
\sup_{x}\, \big(\lvert \partial_x^i w \rvert+ \lvert z \partial_x^i h \rvert\big)<r, \quad i=0, \dots, n-1
\end{equation}
since $w\in H^n$. The proof follows the strategy adopted in the proof of Lemma \ref{NemistkyLemma}, namely we isolate the terms with the pair of functions in the integrands with the highest order of derivatives and we apply the Cauchy-Schwarz inequality to their $L^2$-scalar product. We remark that is fundamental, as for obtaining the bound \eqref{briscola}, that $\partial_x^n w$ appears linearly in the Taylor expansion of $F$. Indeed by this fact it is sufficient to require the condition \eqref{sharp} for $i \le n-1$ and the loss of regularity due to the Sobolev embedding does not force us to require more smoothness on $w$ than $w\in H^n$.
 Eventually we have 
\begin{align*}
\frac{d}{d z}  F(w+z h) &=\int (w_n+z h_n)\,\frac{d}{d z}\sum_{k\geq q} f_{(k-1, 0, \dots, 0, 1)} (w+z h)^k\,dx\\
&+\int h_n \,\sum_{k\geq q} f_{(k-1, 0, \dots, 0, 1)} (w+z h)^k\,dx\\
&+\int \frac{d}{dz}\sum_{k\geq q} \sum_{\substack{\alpha\in\mathcal{J}_n^k,\\ \alpha_n=0}} f_{\alpha} (w+z h)^{\alpha}\,dx
\end{align*}
and these derivatives exist by the analyticity of $f$ on $B_{\mathbb{C}}(0, r)\times \dots \times B_{\mathbb{C}}(0, r)$.
\end{proof}


Now we prove two facts (recall \eqref{seqcos} for the definition of $\Gamma^{(n)}$):
\begin{itemize}
\item the quadratic terms in the expansion of $\Gamma^{(n)}$ have a particular form ;
\item the cubic remainder of the expansion of $\Gamma^{(n)}$ does not contain derivatives of $w$ of order greater than the ones appearing in the quadratic part (see Section \ref{appsec}). We will do that by showing that the coefficient of the quadratic part associated to the monomials containing the highest number of derivatives is non-zero.
\end{itemize}

\medskip


\begin{remark}\label{remarkone}
We point out the following properties of differential polynomial in $\calP_{n}^{k}$.
\begin{itemize}
\item[$(i)$] Let $g_0\in \calP_{n}^{1}$ for some $n$, and recall that $\int w\,dx=0$. Then one has
\[
\int g_0 \,dx=0;
\]
\item[$(ii)$] let $f$ be a polynomial of degree $2$ depending on the derivatives of $w$ of order exactly $n=2k+1$ for $k\in \N$, then 
\[
\int f\,dx=0.
\]
Indeed, by \eqref{definterval}, \eqref{defmonomial} we need to show that
\[
\int (\del_{x}^{k_{1}}w)(\del_{x}^{k_{2}}w)dx=0,
\]
when $k_1+k_2=2k+1$ and at least one between $k_1$ and $k_2$ is $\geq1$. Assume $k_2\geq1$. Hence by integrating by parts
we have, for $\s=1$ or $\s=-1$
\[
\int (-1)^{\s}(\del_{x}^{k}w)(\del_{x}^{k+1}w)dx=
\int (-1)^{\s}\del_{x}[(\del_{x}^{k}w)^{2}]dx=0;
\]
\item[$(iii)$] by the above computations for $n=2k$, $k\geq0$, 
\begin{equation}\label{costpari}
\Gamma^{(n)}(w)=\int g^{(n)}(w)dx,
\end{equation}
for some $g$ belonging to the class $\Sigma_{n+1}^{3}$, 
namely $\Gamma^{(n)}$ has a zero of order three at the origin. \\
On the other hand for $n=2k+1$, $k\geq0$, we simply have that
\begin{equation}\label{costdispari}
\Gamma^{(n)}(w)=\int f_{2}^{(n)}+h^{(n)}(w)dx,\qquad h^{(n)}(w)\in\Sigma^{3}_{n+1}, \quad f_{2}^{(n)}\in \calP_{n+1}^2.
\end{equation}
More precisely $f_2^{(n)}$ has the form
\begin{equation}\label{geneform}
f_{2}^{(n)}=\sum_{p=0}^{n+1}\sum_{k_1+k_2=p}(\del_{x}^{k_1}w)(\del_{x}^{k_2}w)c^{k_1,k_2}_{n}.
\end{equation}

\end{itemize}
\end{remark}

In the following Sections we analyse precisely the form
of $c^{k_1,k_2}_{n}$.

\subsection{Computation of $\Gamma^{(n)}$ for $n$ odd} \label{appsec}

Now we want to derive some explicit expression for the coefficients 
of the quadratic part of the functions $\Gamma^{(n)}$ ($n \in \mathbb{N}$ is odd)
 introduced in \eqref{seqcos}; more precisely, by recalling the definitions of Section \ref{diffpol} and \eqref{costdispari}-\eqref{geneform}, we want to compute 
the coefficients $c_{n}^{k_1,k_2}$ of $f_2^{(n)}$.

\subsubsection{Coefficients of the linear terms} \label{linsubsec}

We begin by computing the coefficients $c_{n}:=c_{n}^{n+1}$ of the linear terms, since this will be useful for the computation of the coefficients of 
the quadratic terms. Consider the recursion relation \eqref{ordbyord1} between 
the $\rho^{(n)}$; since $p \in \Sigma^0_0$ and that 
$p=-\mathtt{c}^{1/3}+\calO(w)$ for small $|w|$, 
the coefficients in front of the leading order linear term of $\rho^{(n+2)}p$ is
 proportional to the coefficient of maximal order of the linear term of $\rho^{(n+2)}$. 
Hence, if we write only the coefficients of maximal order for the linear terms, 
we get
\begin{align}
-c_{n} &= \mathtt{c}^{2/3} \, c_{n+2} -\mathtt{c}^{1/3} 3 c_{n+1}, \nonumber \\
c_{n+2} &= -\mathtt{c}^{-2/3} c_{n}+3 \mathtt{c}^{-1/3} c_{n+1}, n \geq 0. \label{lincoeff}
\end{align}
Now, recall that by \eqref{rho0} and \eqref{rho1} we have that 
$c_0=-\frac{1}{3\mathtt{c}}$ and $c_1=-\frac{2}{9\mathtt{c}^{4/3}}$. 
One can check that 
\begin{align}
c_m = d_1 a^{m} + d_2 b^{m}, \; & m \geq 0, \label{cm} \\
a:= a(\mathtt{c}) = \frac{3+\sqrt{5}}{2\mathtt{c}^{1/3} }, &\; \; b:= b(\mathtt{c}) = \frac{3-\sqrt{5}}{2\mathtt{c}^{1/3} }, \label{ab} \\
d_1 := d_1(\mathtt{c}) = \frac{ -3-\sqrt{5} }{18 \mathtt{c}}, &\; \; d_2 := d_2(\mathtt{c})=\frac{ -3+\sqrt{5} }{18 \mathtt{c}}. \label{d1d2}
\end{align}

\begin{remark} \label{clinprop}
Notice that the following properties hold.
\begin{enumerate}
\item[$(i)$] From \eqref{d1d2} one readily obtains that
\begin{align*}
\lim_{\mathtt{c} \to 0^\pm} d_1(\mathtt{c}) &= \mp \infty, \\
\lim_{\mathtt{c} \to 0^\pm} d_2(\mathtt{c}) &= \mp \infty. 
\end{align*}
The last two limits imply that in the dispersionless limit
\begin{align*}
\lim_{\mathtt{c} \to 0^\pm} c_m &= -(sgn( \mathtt{c} ))^{m} \infty, \; \; m \geq 1.
\end{align*}

\item[$(ii)$] By direct computation one also obtains that:
\begin{enumerate}
\item for any $\mathtt{c}>0$ 
\begin{equation*}
\begin{cases}
c_m<0 &\text{for even} \;  m \geq 2; \\
c_m<0 &\text{for odd} \; m \geq 2, 
\end{cases}
\end{equation*}

\item for any $\mathtt{c}<0$ 
\begin{equation*}
\begin{cases}
c_m>0 &\text{for even} \; m \geq 2; \\
c_m<0 &\text{for odd} \; m \geq 2, 
\end{cases}
\end{equation*}

\item $\lim_{m \to \infty} |c_m| = +\infty$ (respectively, $\lim_{m \to \infty} |c_m| = 0$) for $|\mathtt{c}| < \mathtt{c}^\ast:=\left(\frac{3+\sqrt{5}}{2}\right)^3$ (respectively, for $|\mathtt{c}| > \mathtt{c}^\ast$). 
\end{enumerate}

\end{enumerate}
\end{remark}

\subsubsection{Coefficients of the quadratic terms} \label{quadsubsec}

To determine the coefficients  in front of the quadratic terms containing the maximal number of derivatives in $f_2^{(n)}$, we integrate \eqref{ordbyord1}
\begin{align}
\int \rho^{(n+2)}p^{2} dx &=\int \rho^{(n)}-\rho^{(n)}_{xx}-3\sum_{k_1+k_2=n+1}p\rho^{(k_1)}\rho^{(k_2)}  -\sum_{k_1+k_2+k_3=n}\rho^{(k_1)}\rho^{(k_2)}\rho^{(k_3)} dx \label{eqint1} \\
 &-\int 3\del_{x}(\rho^{(n+1)}p)-3\sum_{k_1+k_2=n}\rho^{(k_1)}\rho^{(k_2)}_x dx. \label{eqint2}
\end{align}
Now, it is easy to see that the integral in \eqref{eqint2} vanishes, since the term $\del_{x}(\rho^{(n+1)}p)$ is a total derivative, and since
\begin{align*} 
\sum_{k_1+k_2=n}\rho^{(k_1)}\rho^{(k_2)}_x &= \sum_{ \substack{k_1+k_2=n \\ k_1>k_2} } \del_x(\rho^{(k_1)}\rho^{(k_2)}). 
\end{align*}
Now, if write $\rho^{(n)} = f_2^{(n)}+h^{(n)}$ and we consider only the coefficients in front of the quadratic terms containing the maximal number of derivatives, the relation \eqref{eqint1}- \eqref{eqint2} reads as
\begin{align} \label{inteq2}
\mathtt{c}^{2/3} \int \sum_{k_1+k_2=n+3} c^{k_1,k_2}_{n+3} (\del_x^{k_1}w)(\del_x^{k_2}w) dx &= 3\mathtt{c}^{1/3} \int \sum_{k_1+k_2=n+1} c_{k_1}c_{k_2} (\del_x^{k_1+1}w)(\del_x^{k_2+1}w) dx,
\end{align}
but since 
\begin{align*}
\int \sum_{k_1+k_2=n+3} c^{k_1,k_2}_{n+3} (\del_x^{k_1}w)(\del_x^{k_2}w) dx &= \sum_{k=0}^{n+3} c^{k,n+3-k}_{n+3} (-1)^{q(k)} \int (\del_x^{\frac{n+3}{2}} w)^2 dx, \\
& q(k):= \frac{n+3}{2}-k, \\
\int \sum_{k_1+k_2=n+1} c_{k_1}c_{k_2} (\del_x^{k_1+1}w)(\del_x^{k_2+1}w) dx &= \sum_{k=0}^{n+1} c_{k} c_{n-k+1} (-1)^{\tilde{q}(k)} \int (\del_x^{\frac{n+3}{2}} w)^2 dx, \\
& \tilde{q}(k):= \frac{n+1}{2}-k,
\end{align*}
we obtain
\begin{align} \label{coeffquadeq}
\sum_{k=0}^{n+3} (-1)^{q(k)} c^{k,n+3-k}_{n+3} &= 3\mathtt{c}^{-1/3} \sum_{k=0}^{n+1} (-1)^{\tilde{q}(k)} c_{k}c_{n-k+1}.
\end{align}

Now we show that the coefficients in front of the quadratic terms containing the maximal number of derivatives do not vanish. We recall that by Remark \ref{remarkone}-$(ii)$ we deal only with $n$ odd.

\begin{prop} \label{coeffquadlemma}
Let $n$ be odd, then, recalling \eqref{coeffquadeq}, we have
\begin{align}
S_n := \sum_{k=0}^{n+1} (-1)^{\tilde{q}(k)} c_{k}c_{n-k+1} &\neq 0.
\end{align}
\end{prop}

\begin{proof}
Consider the right-hand side of Eq. \eqref{coeffquadeq}; by simple calculations
\begin{align} \label{coeffquadeq2}
 S_n &= 2 (-1)^{\frac{n+1}{2}} \left( \sum_{ \substack{ k=0,\ldots,\frac{n-1}{2} \\ k \; \text{even} } } c_{k} c_{n+1-k} - \sum_{ \substack{ k=0,\ldots,\frac{n-1}{2} \\ k \; \text{odd} } } c_{k} c_{n+1-k} \right) +  c_{\frac{n+1}{2}}^2.
\end{align}

First, one can verify explicitly that 
\begin{align*}
S_1 &= -2c_0c_2+c_1^2 = -2(d_1+d_2)(d_1a^2+d_2b^2)+(d_1a+d_2b)^2 \nonumber \\
&= -\frac{14}{81 \mathtt{c}^{8/3} } \neq 0.
\end{align*}

Now we distiguish the two cases $n=4l+3$ ($l \geq 0$) and $n=4l+1$ ($l>0$). 

\emph{Case $n=4l+3$:} first observe that 
\begin{align}
c_{2s} c_{n+1-2s} - c_{2s+1} c_{n-2s} &= d_1 d_2 \left[ a^{2s} b^{n+1-2s} + b^{2s} a^{n+1-2s} - a^{2s+1} b^{n-2s} - b^{2s+1} a^{n-2s} \right] \nonumber \\
&= d_1 d_2 \left[ a^{2s} b^{n-2s} (b-a) + b^{2s} a^{n-2s} (a-b) \right] \nonumber \\
&= d_1 d_2 \left[ b^n \left(\frac{a}{b}\right)^{2s} (b-a) - a^n \left( \frac{b}{a} \right)^{2s} (b-a) \right] \nonumber \\
&= d_1 d_2 (b-a) \left[ b^n \left( \frac{a}{b} \right)^{2s} - a^n \left( \frac{b}{a} \right)^{2s} \right], \label{prodcoeff}
\end{align}
so in this case we have 
\begin{align*}
\sum_{s=0}^{\frac{n-3}{4}} (c_{2s} c_{n+1-2s} - c_{2s+1} c_{n-2s}) &= d_1 d_2 (b-a) \left[ b^n \frac{ 1- \left( \frac{a^2}{b^2} \right)^{\frac{n+1}{4}} }{ 1-a^2/b^2 } - a^n \frac{ 1- \left( \frac{b^2}{a^2} \right)^{\frac{n+1}{4}} }{ 1-b^2/a^2 } \right] \\
&= d_1 d_2 (b-a) \left( b^{n+2-\frac{n+1}{2}} \frac{ b^{\frac{n+1}{2}} - a^{\frac{n+1}{2}} }{b^2-a^2} + a^{n+2-\frac{n+1}{2}} \frac{ a^{\frac{n+1}{2}}-b^{\frac{n+1}{2}} }{b^2-a^2} \right) \\
&=\frac{d_1 d_2}{a+b} ( b^{\frac{n+3}{2}} - a^{\frac{n+3}{2}} )( b^{\frac{n+1}{2}} - a^{\frac{n+1}{2}} ) \\
&\stackrel{n=4l+3}{=} \frac{d_1d_2}{a+b} (b^{2l+3}-a^{2l+3})(b^{2l+2}-a^{2l+2}),
\end{align*}

and the thesis is equivalent to 
\begin{align} \label{case3ineq}
2 \frac{d_1 d_2}{a+b} (a^{2l+3}-b^{2l+3})(a^{2l+2}-b^{2l+2}) + (d_1 a^{2l+2}+d_2 b^{2l+2})^2 &\neq 0
\end{align}
In order to verify \eqref{case3ineq} we observe that
\begin{align*}
(a^{2l+3}-b^{2l+3})(a^{2l+2}-b^{2l+2}) &= a^{4l+5} \left[ 1 - \left(\frac{b}{a}\right)^{2l+2} - \left(\frac{b}{a}\right)^{2l+3} + \left(\frac{b}{a}\right)^{4l+5} \right] =: a^{4l+5} \alpha_l,
\end{align*}
where $(\alpha_l)_{l \in \mathbb{N}}$ is an increasing sequence of positive numbers (which do not depend on $\mathtt{c}$) satisfying $\lim_{l \to \infty} \alpha_l=1$; similarly, we have
\begin{align*}
(d_1 a^{2l+2} + d_2 b^{2l+2})^2 &= a^{4l+4} d_1^2 \left[ 1 + \frac{d_2}{d_1} \left( \frac{b}{a} \right)^{2l+2} \right]^2 =: a^{4l+4} d_1^2 \beta_l,
\end{align*}
where $(\beta_l)_{l \in \mathbb{N}}$ is a decreasing sequence of positive numbers (which do not depend on $\mathtt{c}$) satisfying $\lim_{l \to \infty} \beta_l=1$. Since $2a \frac{d_1 d_2}{a+b}>0$ by \eqref{ab} and \eqref{d1d2}, we have that the left-hand side in \eqref{case3ineq} is given by
\begin{align*}
a^{4l+4} \left( 2a \frac{d_1d_2}{a+b}  \alpha_l +  d_1^2 \beta_l \right) &= 
\frac{1}{\mathtt{c}^{2+(4l+4)/3}} \left(\frac{3+\sqrt{5}}{2}\right)^{4l+4} \left( \frac{3+\sqrt{5}}{243} \alpha_l + \frac{7+3\sqrt{5}}{162} \beta_l \right) > 0.
\end{align*}

\emph{Case $n=4l+1$ ($l \geq 1$):} by arguing as in \eqref{prodcoeff}, we get
\begin{align*}
\sum_{s=0}^{l-1} (c_{2s} c_{n+1-2s} - c_{2s+1} c_{n-2s}) &= \frac{d_1 d_2}{a+b} (b^{2l+2}-a^{2l+2})(b^{2l+1}-a^{2l+1}),
\end{align*}
and the thesis is equivalent to the following inequality,
\begin{align} \label{case1ineq}
2 \frac{d_1 d_2}{a+b} (b^{2l+2}-a^{2l+2})(b^{2l+1}-a^{2l+1}) + 2 (d_1 a^{2l}+d_2 b^{2l})(d_1 a^{2l+2} + d_2 b^{2l+2} ) - (d_1 a^{2l+1} + d_2 b^{2l+1})^2 &\neq 0.
\end{align}
In order to verify \eqref{case1ineq} we observe that
\begin{align*}
& 2 \frac{d_1 d_2}{a+b} (a^{2l+2}-b^{2l+2})(a^{2l+3}-b^{2l+3}) \\
&= 2 \frac{d_1 d_2}{a+b} a^{4l+3} \left[ 1 - \left( \frac{b}{a} \right)^{2l+1} - \left( \frac{b}{a} \right)^{2l+2} + \left( \frac{b}{a} \right)^{4l+3} \right] =: 2 \frac{d_1 d_2}{a+b} a^{4l+3} \gamma_l,
\end{align*}
where $(\gamma_l)_{l \geq 1}$ is a increasing sequence of positive numbers (which do not depend on $\mathtt{c}$) satisfying $\lim_{l \to \infty} \gamma_l=1$; similarly,
\begin{align*}
& 2(d_1 a^{2l} + d_2 b^{2l})(d_1 a^{2l+2} + d_2 b^{2l+2}) \\
&= 2d_1^2 a^{4l+2} \left[ \left(1 + \frac{d_2}{d_1} \left(\frac{b}{a}\right)^{2l} \right) \left(1 + \frac{d_2}{d_1} \left(\frac{b}{a}\right)^{2l+2} \right) \right] =: 2d_1^2 a^{4l+2} \delta_l,
\end{align*}
where $(\delta_l)_{l \geq 1}$ is an decreasing sequence of positive numbers (which do not depend on $\mathtt{c}$) satisfying $\lim_{l \to \infty} \delta_l=1$, and
\begin{align*}
(d_1 a^{2l+1} + d_2 b^{2l+1})^2 &= a^{4l+2} d_1^2 \left( 1 + \frac{d_2}{d1} (b/a)^{2l+1} \right)^2 =: a^{4l+2} d_1^2 \epsilon_l,
\end{align*}
where $(\epsilon_l)_{l \geq 1}$ is a decreasing sequence of positive numbers (which do not depend on $\mathtt{c}$) such that $\lim_{l \to \infty} \epsilon_l=1$.
Since $2a \frac{d_1 d_2}{a+b}>0$ by \eqref{ab} and \eqref{d1d2}, we have that the left-hand side in \eqref{case3ineq} is given by
\begin{align*}
  a^{4l+2} \left( 2a \frac{d_1 d_2}{a+b} \gamma_l + 2 d_1^2 \delta_l - d_1^2\epsilon_l \right) &= \frac{1}{\mathtt{c}^{2+(4l+2)/3}} \left(\frac{3+\sqrt{5}}{2}\right)^{4l+2} \left(\frac{3+\sqrt{5}}{243} \gamma_l + \frac{7+3\sqrt{5}}{81} \delta_l - \frac{7+3\sqrt{5}}{162} \epsilon_l \right) \\
&\neq 0.
\end{align*}
\end{proof}

\begin{remark}\label{displess}
By arguing as in the proof of the above proposition, one can also show that for any odd number $n$
\begin{equation*}
\lim_{\mathtt{c} \to 0} |S_n| = + \infty,
\end{equation*}
which reflects the fact that the present approach cannot be applied to the dispersionless DP equation \eqref{dispersionless}.
\end{remark}

By Proposition \ref{coeffquadlemma} we have that the number of derivatives appearing in the quadratic part is greater or equal than the one appearing in the cubic remainder.

\subsection{Proof of Theorem \ref{costantiMotoDP}}

The proof of Theorem \ref{costantiMotoDP} is based on the following reasoning.
We recall that, by the discussion in Section \ref{appsec}, 
we have constructed  sequence of constants of motion $\Gamma^{(n)}(w)$ with $n\geq0$ of the form
\begin{equation}
\Gamma^{(n)}(w)=\int_{\mathbb{T}}\rho^{(n)}(w)dx,
\end{equation}
where $\rho^{(n)}\in \Sigma^{0}_{n+1}$ (see Lemma \ref{lemmaron}) are defined iteratively by \eqref{ordbyord}, \eqref{ordbyord1}.
We recall also (see \eqref{costpari}, \eqref{costdispari})
that, for $n$ even 
\begin{equation}\label{costpari10}
\Gamma^{(n)}(w)=\int_{\mathbb{T}}\rho^{(n)}(w)dx=\int_{\mathbb{T}}g^{(n)}(w)dx, \quad g^{(n)}\in \Sigma^{3}_{n+1},
\end{equation}
while for $n$ odd
\begin{equation}\label{costdispari10}
\Gamma^{(n)}(w)=\int_{\mathbb{T}}\rho^{(n)}(w)dx=\int_{\mathbb{T}}f_2^{(n)}(w)+h^{(n)}(w)dx, \quad h^{(n)}\in \Sigma^{3}_{n+1},
\end{equation}
and $f_{2}^{(n)} \in \calP_{n+1}^2$ as in \eqref{geneform}. 

\begin{prop}\label{seqvera}
For any $N\in \N$, $N\geq 2$ there are $r=r(N)>0$, $c_1=c_1(N)>0$, $c_2=c_{2}(N)>0$ and a sequence of functions 
\begin{equation}\label{funz}
F_n :  H^N \to \R,\quad 1\leq n\leq N,
\end{equation}
analytic on $B_{H^N}(0,  \lvert \mathtt{c} \rvert /2)$
with the following properties.
\begin{itemize}
\item[(i)] For any $0\leq n\leq N$, we have
\begin{equation}\label{uffa1}
F_n(w)=\int g^{(n)}dx, \quad g^{(n)}\in \Sigma_{n+1}^{3}, \;\;\; n=2k, \quad {\rm for \;\; some} \quad k\in \N,
\end{equation}
and 
\begin{equation}\label{uffa2}
F_n(w)=\int \Big(q_2^{(n)}+q_3^{(n)} \Big) dx, \quad  q_3^{(n)}\in \Sigma_{n+1}^{3}, \;\;\; n=2k+1, \quad {\rm for \;\; some} \quad k\in \N,
\end{equation}
where $q_{2}^{(n)} \in \calP^{2}_{n+1}$ and
\begin{equation}\label{uffa22}
\int q_{2}^{(n)}(w)\,dx=\int (\del_{x}^{k+1}w)^{2}\,dx.
\end{equation}
\item[(ii)] For $n=2k$, $k\in \N$, we have
\begin{equation}\label{uffa4}
|F_n(w)|\leq c_1\|w\|_{H^{n+1}}^{3}, \quad \forall \; 0\leq n\leq N, \;\;  w\in B_{H^{N}}(0,r),
\end{equation}
\item[(iii)] for $n=2k+1$, $k\in \N$, we have
\begin{equation}\label{uffa5}
|F_n(w)-\int q_2^{(n)}dx|\leq c_{2} \|w\|_{H^{n+1}}^{3}, \quad \forall \; 0\leq n\leq N, \;\;  w\in B_{H^{N}}(0,r).
\end{equation}
\end{itemize}

Finally, let $u(t,x)$ be the solution of \eqref{DP} with $u(0,x)=u_0(x)\in H^{N+2}$ such that $\|u_0\|_{H^{N+2}}\leq \e$. 
Define $w(t,x)=(1-\del_{xx})u(t,x)$. Then, as long as $\|u(t,\cdot)\|_{H^{N+2}}\leq r/2$,  we have that
\begin{equation}\label{uffa3}
\frac{d}{dt}F_n(w)=0, \quad 0\leq n\leq N.
\end{equation}

\end{prop}

 \begin{proof}
 First of all note that  for $0\leq s\leq N$
 \[
 \|w\|_{H^{s}}\leq \|u\|_{H^{s}}+\|u\|_{H^{s+2}}.
 \]
 This implies that as long as $\|u(t,\cdot)\|_{H^{N+2}}\leq r/2$ one has
 $\|w\|_{H^{s}}\leq r$.

For $n\in \N$ even we define $F_n(w):=\Gamma^{(n)}(w)$, hence \eqref{uffa1} holds by \eqref{costpari10}. The bound \eqref{uffa4} follows by Proposition \ref{LemmaStimeTaylor}.\\
For $n$ odd we construct the functions $F_n$ iteratively, by putting the quadratic parts of the functions $\Gamma^{(n)}$ with $n$ odd in a \emph{triangular} form. \\
First we observe that for any $n$ odd (see \eqref{geneform})
\begin{equation}\label{real}
\int f_2^{(n)}\,dx= \sum_{\substack{0\le p\leq n+1,\\ p\equiv 0 (2)} } \sum_{k_1+k_2=p} c^{k_1 k_2}_n \,\int (\partial_x^{k_1} w)\,(\partial_x^{k_2} w)\,dx=\sum_{i=0}^{(n+1)/2} d_{n}^i\,\int (\partial_x^{i} w)^2\,dx
\end{equation}
for some coefficients $d_n^i$ obtained after the integration by parts and by Lemma \ref{coeffquadlemma} 
\begin{equation}\label{REM}
d_n^{(n+1)/2}\neq 0 \qquad \forall n.
\end{equation}
Now let us show the first step of the triangularization. We write the quadratic part of $\Gamma^{(1)}$ as
\[
\int f_2^{(1)}\,dx=d_1^0 \int w^2\,dx+d_1^1\,\int w_x^2\,dx.
\]
Recall the definition of the constants of motion $M_1$ in \eqref{nonleho}. Its quadratic part is, by \eqref{emme}, \eqref{p} and Remark \ref{remarkone},
\[
M_1^{(2)}=\frac{1}{9} \int w^2\,dx.
\]
We remark that $\Gamma^{(1)}$, $M_1$ and their linear combinations are constants of motion. We define (recall \eqref{REM})
\[
F_1:=\frac{1}{d_1^1}\left(\Gamma^{(1)}-9 d_1^0 M_1\right)
\]
so that its quadratic part reads as
\begin{equation}\label{passozero}
\int q_2^{(1)}\,dx:=\int w_x^2\,dx
\end{equation}
and for the cubic part we have
\[
\int q_3^{(1)}\,dx=\frac{1}{d_1^1}h^{(1)}-9 d_1^0 M_1^{(\geq 3)}.
\]
Since by Lemma \ref{polem} $q_3^{(1)}\in\Sigma_{2}^3$, by Proposition \ref{LemmaStimeTaylor} the bound \eqref{uffa5} holds for $F^{(1)}$.\\
We define iteratively for $n=2 k+1$, $k\geq 1$
\begin{equation}\label{iteration}
F_{2 k+1}:=\frac{1}{d_{2 k+1}^{k+1}}\left( \Gamma^{(2 k+1)}-9 d_{2k+1}^0 M_1-\sum_{i=1}^{k} d_{2k+1}^i F_{2 i+1} \right)
\end{equation}
and \eqref{uffa2} follows easily by Lemma \ref{polem}.
Now we prove \eqref{uffa22} by induction on $k\geq 0$. We just showed the basis of the induction in \eqref{passozero} ($k=0$ ). Now suppose that
\[
\int q_2^{(2 k+1)}\,dx=\int (\partial_x^{k+1} w)^2\,dx.
\] 
Then by \eqref{iteration} 
\begin{align*}
d^{k+2}_{2k+3}\int q_2^{(2 k+3)}\,dx &=\int f_2^{(2k +3)}\,dx- d_{2k+3}^0 \int w^2\,dx-\sum_{i=1}^{k+1} d_{2k+3}^i \int (\partial_x^i w)^2\,dx\stackrel{(\ref{real})}{=}d^{k+2}_{2k+3}\int (\partial_x^{k+2} w)^2\,dx,
\end{align*}
which proved the claim. The bound \eqref{uffa5} follows trivially by triangle inequality and Proposition \ref{LemmaStimeTaylor}.\\
Obviously the $F_{2k +1}$ defined in \eqref{iteration} are constants of motion because they are linear combinations of conserved quantities.

 \end{proof}
\begin{proof}[{\bf Proof of Theorem \ref{costantiMotoDP}}] 
We set $K_1(u):=M_1(u)$ where $M_1$ is defined in \eqref{nonleho}. 
For any $n\geq2$ of the form $n=k+1$ we set 
\[
K_{n}(u)=F_{2k-1}(w),
\]
where $F_{2k-1}(w)$ is given in Proposition \ref{seqvera}.\\
The item $(0)$ is verified since the functions $F_{n}$ constructed in Proposition \eqref{costantiMotoDP} satisfy \eqref{uffa3}.\\
The item $(i)$ and $(ii)$ follows by item $(i)$ of Proposition \eqref{costantiMotoDP}.\\
For the item $(iii)$, \eqref{stimecostanti} follows by \eqref{uffa5} and \eqref{equivalenzaNorma} by \eqref{uffa2}, the definitions \eqref{spazio}, \eqref{spazioR} and by considering the equivalent norm $\lVert u \rVert_{L^2}+\lVert \partial_x^s u \rVert_{L^2} \approx \lVert u \rVert_{H^s}$.


\end{proof}

\section{Global well-posedness near the origin}\label{GWPDP}

In this Section we give the proof of Theorem \ref{esistenzaglobale} by using Theorem \ref{costantiMotoDP}.
First we state the following local well-posedness result for Eq. \eqref{DP}.

\begin{prop}[{\bf Local existence}]\label{esistesol2}
For any $s > 3/2$ and for any $u_0\in H^s$  there exist 
$\bar{T}=\bar{T}(\|u_0\|_{H^s})$ and a unique solution $u(t,x)$ of \eqref{DP} 
with initial condition $u(0,x)=u_0(x)$ defined for $t\in [-\bar{T},\bar{T}]$ 
belonging to the space $C([-\bar{T},\bar{T}],{{H}}^{s})$,  such that
\begin{align}
\|u(t)\|_{H^s} &\leq 2 \|u_0\|_{H^s}, \; \; |t| \leq \bar{T} \leq \frac{1}{ 2C_s \|u_0\|_{H^s} }, \label{estnormsol}
\end{align}
for some constant $C_s>0$ depending only on $s$.
\end{prop}

The proof of the above result can be found in the Appendix \ref{appendA}. 
It is based on a Galerkin-type approximation method, and follows closely the 
argument reported in \cite{HH} for the dispersionless DP equation 
\eqref{dispersionless}. 

\begin{remark} \label{smalldata}
Proposition \ref{esistesol2} implies that for any $s>3/2$ there exists $r(s)>0$ 
such that for any $u_0 \in B_{H^s}(0,r(s))$ Eq. \eqref{DP} with initial datum 
$u_0$ admits a solution $u(t,x)$ such that $\|u(t)\|_{H^s} \leq 2\|u_0\|_{H^s}$ 
for $|t| \leq \frac{1}{2C_s r(s)}$.
\end{remark}

The proof of Theorem \ref{esistenzaglobale} is based on the following bootstrap argument.

Consider the function $u(t,x)$ solution of \eqref{DP}, defined on some interval  $t\in[-{T},{T}]$ with $0<T\leq \bar{T}$
with initial datum $\|u(0;\cdot)\|_{H^{s}}\leq r_0$ 
given by Proposition \ref{esistesol2}. 

\vspace{0.8em}

By choosing $n:=n(s)=[s]-1$, if $r_0$ small enough, we have by item $(iii)$ of Theorem \ref{costantiMotoDP} and by Proposition \ref{esistesol2} that
\begin{align} \label{bound1}
\|u(t,\cdot)\|_{H^{s}}^{2}\leq \|u(\cdot)\|_{L^{2}}^{2}+\tilde{c} |K_{n}^{(0)}(u(t,\cdot))| &\leq \|u(\cdot)\|_{L^{2}}^{2}+\tilde{c} |K_{n}(u(t,\cdot))|+ \tilde{c} C\|u(t,\cdot)\|_{H^{s}}\|u(t,\cdot)\|^{2}_{H^{s}},
\end{align}
with $\tilde{c},C$ given by Theorem \ref{costantiMotoDP}.
Since 
\begin{align*}
K_1^{(0)}(u(t,\cdot)) &= \int (u-u_{xx})^2 dx \\
&= \int u^2 dx + 2 \int u_x^2  dx + \int u_{xx}^2 dx
\end{align*}

by \eqref{stimecostanti} we have that for $r_0$ small enough
\begin{align} \label{boundL2}
\|u(t,\cdot)\|^2_{L^2} &\leq C |K_1(u(t,\cdot))|.
\end{align}

Since $K_1(u),K_{n}(u)$ are constant of motion, we have
\begin{equation}\label{valesempre}
|K_{1}(u(t,\cdot))|+|K_{n}(u(t,\cdot))|\leq |K_{1}(u(0,\cdot))|+|K_{n}(u(,\cdot))|\stackrel{(\ref{equivalenzaNorma}),(\ref{stimecostanti})}{\leq} \ka(s)\|u(0,\cdot)\|^{2}_{H^{s}}\leq \ka(s)r_0^{2},
\end{equation}
for some $\ka(s)>0$
depending only on $s$. The bound \eqref{bound1} reads
\begin{equation}\label{bound1bis}
\|u(t,\cdot)\|_{H^{s}}^{2}\leq\tilde{c}\ka(s)r_0^{2}+\tilde{c} C\|u(t,\cdot)\|_{H^{s}}\|u(t,\cdot)\|^{2}_{H^{s}}.
\end{equation}

Now let  $\widehat{T}$ be the supremum of those $T$ such that the solution $u(t,x)$ is defined on $[-T,T]$ and 
\begin{equation}\label{assurdo}
\sup_{t\in [-T,T]}\|u(t,\cdot)\|^{2}_{H^{s}}\leq  Q(s)r^2_0, 
\end{equation}
where $Q(s)\geq 4 \ka(s)\tilde{c}$, with $\ka(s)$  given in \eqref{valesempre} and $\tilde{c}$ given by Theorem \ref{costantiMotoDP}.
For $t\in [-\hat{T},\hat{T}]$ we deduce, by \eqref{bound1bis}, that
\begin{equation}\label{bound1tris}
\|u(t,\cdot)\|_{H^{s}}^{2}\leq \tilde{c}\ka(s)r_0^{2}+\tilde{c}C\sqrt{Q(s)} r_0\|u(t,\cdot)\|^{2}_{H^{s}}.
\end{equation}
Hence, if we take $r_0$ sufficiently small such that $\tilde{c}C\sqrt{Q(s)} r_0\leq 1/2$, we obtain
\begin{equation}\label{speriamobene}
\|u(t,\cdot)\|_{H^{s}}^{2}\leq 2\tilde{c} r_0^{2}<Q(s)r_0^{2}.
\end{equation}
Of course estimate \eqref{speriamobene} leads to the  contradiction of the fact that $\widehat{T}$ is the supremum.
Since estimate \eqref{speriamobene} does not depend on $\hat{T}$, we must have $\widehat{T}=+\infty$, which implies Theorem \ref{esistenzaglobale}.

\section{Birkhoff resonances}\label{proofteoBirk}

In this Section we prove Theorem \ref{teoBirk}. First
we need some preliminaries definitions and results to show the formal Birkhoff normal form procedure. Recall the definitions given in Section \ref{prelimi} for the space of formal polynomials.

\begin{defi}{(\bf Poisson brackets)}
Let $P=\sum_{\alpha\in\mathcal{I}_n} P_{\alpha} u^{\alpha}$ and $Q=\sum_{\beta\in\mathcal{I}_m} Q_{\beta} u^{\beta}$ (recall \eqref{definterval}) two formal homogenous polynomials.
We define $\{\,\cdot\,,\, \cdot\, \}\colon \mathscr{P}^{(n)}\times \mathscr{P}^{(m)}\to \mathscr{F}$
\begin{equation}\label{poisson}
\{ P, Q\}:=\sum_{\alpha\in\mathcal{I}_n, \beta\in\mathcal{I}_m} P_{\alpha} Q_{\beta}\,\{ u^{\alpha}, u^{\beta}\}=\sum_{\alpha\in\mathcal{I}_n, \beta\in\mathcal{I}_m} \sum_{j} \big((-\mathrm{i}) \omega(j) \alpha_j \beta_{-j} \big) P_{\alpha} Q_{\beta} \,u^{\alpha+\beta-\mathtt{e}_j-\mathtt{e}_{-j}}
\end{equation}
where $\mathtt{e}_k$ is the element of $\mathbb{N}^{\mathbb{Z}}$ with all components equal to zero except for the $k$-th one, which is equal to $1$.
\end{defi}
In the following lemma we prove that the above definition is well posed. We point out that the assumption of zero momentum (see \eqref{definterval}) is a key ingredient for the proof.
\begin{lem}
Given $P\in\mathscr{P}^{(n)}$, $Q\in\mathscr{P}^{(m)}$ we have that
\begin{itemize}
\item[$(i)$] the sum $P+Q\in\mathscr{P}^{({\max\{m, n\}})}$.
\item[$(ii)$] the Poisson bracket $\{ P, Q\}\in\mathscr{P}^{(n+m)}$.
\end{itemize}
\end{lem}
\begin{proof}
The item $(i)$ is trivial. We prove item $(ii)$. We write $P=\sum_{\alpha\in\mathcal{I}_n} P_{\alpha} u^{\alpha}$ and $Q=\sum_{\beta\in\mathcal{I}_m} Q_{\beta} u^{\beta}$.\\
Recalling \eqref{poisson}, we have to prove the following claim: given $\gamma\in\mathcal{I}_{n+m-2}$ there is only a finite number of $\alpha, \beta, j$ such that 
\begin{equation}\label{bob}
\gamma=\alpha+\beta-\mathtt{e}_j-\mathtt{e}_{-j}, \quad \alpha_j \beta_{-j}\neq 0,
\end{equation}
where we denoted by $\mathtt{e}_j$ the element of $\mathbb{N}^{\mathbb{Z}}$ with all components zero except for the $j$-th, which is $1$.
Indeed, if this holds, there exists a sequence $(R_{\gamma})_{\gamma\in\mathcal{I}_{n+m-2}}$ of complex numbers such that
\[
\{ P, Q\}=\sum_{\gamma\in \mathcal{I}_{m+n-2}} R_{\gamma}\,u^{\gamma}.
\]
First we observe the following: given $\gamma\in\mathcal{I}_{k}$, for some $k\geq 2$, there exist only finitely many couples $(a, b)\in\mathbb{N}^{\mathbb{Z}}\times \mathbb{N}^{\mathbb{Z}}$ such that we can decompose $\gamma=a+b$. Note that $\mathcal{M}(\gamma)=\mathcal{M}(a)+\mathcal{M}(b)=0$. We call $\mathtt{M}^{\gamma}:=\sum_{j>0} j \gamma_j$ and we observe that, for any choice of $a$ and $b$, we have $\lvert \mathcal{M}(a) \rvert\le \mathtt{M}^{\gamma}$.\\
We can choose $a$ and $b$ such that $\mathcal{M}(a)=j$, for instance,
\[
a=\alpha-\mathtt{e}_{-j}, \quad b=\beta-\mathtt{e}_{j}.
\]
hence $\lvert j \rvert\le \mathtt{M}^{\gamma}$. So given $\gamma$ there is only a finite number of $j$ for which  \eqref{bob} holds.\\
Now we prove that, given $\gamma$ and $j$, there is only a finite number of $\alpha$ and $\beta$ such that \eqref{bob} holds and so the claim is proved. We have $\gamma+\mathtt{e}_j+\mathtt{e}_{-j}=\alpha+\beta$. Hence we have to split the left-hand side in two elements of $\mathbb{N}^{\mathbb{Z}}$ with zero momentum. If $\gamma=a+b$ with $(a, b)\in\mathbb{N}^{\mathbb{Z}}\times \mathbb{N}^{\mathbb{Z}}$, then there is only a finite number of choices for $a$ and $b$, which are
\[
a=\alpha-\mathtt{e}_{\pm j}, \quad b=\beta-\mathtt{e}_{\mp j}.
\]

\end{proof}

\paragraph{Adjoint action and quadratic Hamiltonians.}

Let $G\in\mathscr{P}^{(m)}$, $m\geq 0$. We define the \emph{adjoint action} of $G$ as 
\[
\mathrm{ad}_G \colon \mathscr{P}^{(n)}\to \mathscr{P}^{(m+n)} \subset \mathscr{F}^{(\geq m)},\qquad
\mathrm{ad}_G[P]:=\{ G, P\},
\]
then we extend it to the entire $\mathscr{F}$ by setting
\[
\mathrm{ad}_G [P]:=\big(\mathrm{ad}_G [P^{(n)}] \big)_{n\geq 0}
\]
with
\[
\Pi^{(d)} \mathrm{ad}_G[P]=
\begin{cases}
\sum_{n=d-m} \{ G^{(m)}, P^{(n)} \} \quad \mbox{if}\,\, d\geq m,\\
0 \qquad \mbox{otherwise}
\end{cases}
\]
for $d\geq 0$. Note that $\Pi^{(d)} \mathrm{ad}_G[P]\in\mathscr{P}^{(d)}$ since the above sum is finite, so $\mathrm{ad}_G$ is well defined on $\mathscr{F}$.\\
We define the kernel and range of $\mathrm{ad}_G$ as
\[
Ker(G):=\{ F\in\mathscr{F} : \{ F, G\}=0\}, \quad Rg(G):=\{ F\in\mathscr{F} : \{ F, G\}\neq 0\}.
\]
Consider a quadratic Hamiltonian $G$ in diagonal form, $G=\sum_{j} \lambda(j) \lvert u_j \rvert^2$, $\lambda(j)\in\mathbb{C}$. Given $\alpha\in\mathbb{N}^{\mathbb{Z}}$ we define the associated $G$-divisor as (recall \eqref{dispersionLaw})
\begin{equation}\label{Gdivisor}
\Omega_G(\alpha)=\sum_j \omega(j) \lambda(j) \alpha_j
\end{equation}
and we have
 \[
 \{ G(u), u^{\alpha}\}=\Big(-i \sum_{j} \omega(j)\lambda(j) \alpha_j\Big) u^{\alpha}=-i \Omega_{G}(\al)u^{\al}.
 \]
 \begin{defi}
We define  $\Pi_{{\rm Ker}(G)}$ as the  projector on the kernel of the adjoint action ${\rm ad}_{G}[\cdot]$, i.e.
\[
\forall \,\al\in \mathcal{I}_{n}\quad \Pi_{{\rm Ker}(G)}(u^{\al}):=\left\{
\begin{aligned}
&u^{\al}\;\;{\rm if}\;\; \Omega_G(\al)\neq0,\\
&0
\;\;\;\; {\rm if}\;\; \Omega_G(\al)=0.
\end{aligned}\right.
\]
We define the projector on the range of the adjoint action as $\Pi_{{\rm Rg}(G)}:=\mathrm{I}-\Pi_{{\rm Ker}(G)}$
where  $\mathrm{I}$ is the identity.
We define the action of $\Pi_{{\rm Rg}(G)}$ and $\Pi_{{\rm Ker}(G)}$ on any Hamiltonian $H\in \mathscr{P}^{(n)}$ by linearity.
\end{defi}

\paragraph{Exponential map and Lie transformation.}
Let $G\in\mathscr{P}^{(m)}$ with $m\geq 1$. We note that for $k\geq 0$
\[
\mathrm{ad}^k_G\colon \mathscr{P}^{(n)}\to \mathscr{P}^{(n+k m)}\subset \mathscr{F}^{(\geq n+k m)}
\]
We define the exponential map $e^{\{G, \cdot}\colon \mathscr{F}\to \mathscr{F}$ as
\[
e^{\{G, \cdot} P:=\Big(\sum_{k\geq 0} \frac{\mathrm{ad}^{k}_G[P^{(n)}]}{k!}\Big)_{n\geq 0}
\]
and it is well-defined since
\[
\Pi^{(d)} e^{\{G, \cdot} P=
\sum_{(n, k)\,:\,n+k m=d} \frac{\mathrm{ad}^k_G[P^{(n)}] }{k!}
\]
but there is only a finite number of couples $(n, k)\in\mathbb{N}^2$ such that $n+m k=d$ ($d \geq 0$ and $m\geq 1$ are fixed). Hence $\Pi^{(d)} e^{\{G, \cdot} P\in\mathscr{P}^{(d)}$.

\smallskip

Let $\chi\in\mathscr{P}^{(n)}$ and $H\in\mathscr{F}$ be two Hamiltonians. We call $\Phi_{\chi}^t$ the flow of $\chi$, namely
\[
\begin{cases}
\dfrac{d}{d t} \Phi_{\chi}^t(u)=J \nabla \chi(u),\\[2mm]
\Phi_{\chi}^0(u)=u.
\end{cases}
\]
We have
\[
\frac{d^k}{d t^k} \Big( H \circ \Phi_{\chi}^t\Big)=\mathrm{ad}_{\chi}^k[H]\circ \Phi_{\chi}^t
\]
Then by expanding in the (formal) Taylor series at $t=0$ we get
\[
H\circ \Phi_{\chi}^t:=\sum_{k\geq 0} \frac{\mathrm{ad}^k_{\chi}[H]}{k!}\,t^k.
\]
We call Lie transformation the map at time one $\Phi_{\chi}^{t=1}=:\Phi_{\chi}$ and we have
\[
H\circ \Phi_{\chi}:=\sum_{k\geq 0} \frac{\mathrm{ad}^k_{\chi}[H]}{k!}=e^{\{ \chi, \cdot} H.
\]

\subsection{Proof of Theorem \ref{teoBirk}}

Let $H\in\mathscr{P}^{(\le n)}$ be a Hamiltonian. The goal of the formal Birkhoff normal form procedure is to construct a change of coordinates $\Phi$ which puts the Hamiltonian $H$ in a Birkhoff normal form of some order (recall \eqref{BirkhoffFormN}). This algorithm consists of different steps.
If we denote by $\chi_k\in\mathscr{P}^{(k)}$ the generators of the Birkhoff transformation at the step $k$ and with
\[
H_0:=H, \quad H_k:=e^{\{ \chi_k}\cdots e^{\{\chi_1, \cdot} H \quad k>0,
\]
then we have to choose $\chi_{k+1}$ such that
\begin{equation}\label{beam}
 \Pi^{(k+1)}e^{\{ \chi_{k+1}, \cdot} H_k=\Pi_{Ker(H^{(0)})} \Pi^{(k+1)} H_k=\Pi_{Ker(H^{(0)})}  H_k^{(k+1)}
\end{equation}
The right-hand side of \eqref{beam} contributes to the normal form of $H_{k+1}$. We want to show that this homogenous polynomial of degree $k+3$ is supported on $\mathcal{N}_{k+1}^*$ (recall Definition \ref{defrisonanza}).

\medskip

%

First we prove that, given a finite set of Hamiltonians in involution (namely which pairwise commute) and fixed $N$, there exists a Birkhoff transformation $\Phi_N$ which puts all these Hamiltonians in Birkhoff normal form of order $N$ according to Definition \ref{defBirkformN}.  It is sufficient to prove that for two commuting Hamiltonians.

\begin{lem}\label{lemmabellissimo}
Consider $H, K\in\mathscr{F}$ two commuting Hamiltonians. For any $N$ there exist, at least formally, a change of coordinates $\Phi_N$ such that
\begin{equation}
H\circ \Phi_N=H^{(0)}+Z_{N}+R_N, \quad K\circ \Phi_N=K^{(0)}+W_{N}+Q_N
\end{equation}
where $Z_N, W_N\in\mathcal{P}^{(N)}$ commuting with $H^{(0)}$ and $K^{(0)}$. $R_N, Q_N\in\mathscr{F}^{(\geq N+1)}$.
\end{lem}
\begin{proof}
We argue by induction on the number of steps $N$. For $N=0$ it is trivial since $\Phi_0$ is the identity map.\\
Suppose that we have performed $N$ steps. By the fact that $\{ H, K \}=0$ then $\{ H, K\}\circ \Phi_N=0$ and so, at each order, we have
\begin{align*}
&\{ H^{(0)}, K^{(0)}\}=0,\\
&\{ H^{(0)}, W_N\}+\{ Z_N, K^{(0)}\}+\Pi^{(\le N)}\{  Z_N, W_N\}=0,\\
&\Pi^{(N+1)} \{  Z_N, W_N\}+\{ H^{(0)}, Q_N^{(N+1)}\}+\{ R_N^{(N+1)}, K^{(0)} \}=0,\\
&\dots
\end{align*}
By the inductive hypothesis $W_N, Z_N\in Ker(H^{(0)})\cap Ker(K^{(0)})$, hence $\{ H^{(0)}, W_N\}=\{ Z_N, K^{(0)}\}=0$ and 
\begin{equation}\label{identita}
\{ H^{(0)}, Q_N^{(N+1)}\}+\{ R_N^{(N+1)}, K^{(0)} \}=0
\end{equation}
since $\{ H^{(0)}, Q_N^{(N+1)}\}\in Rg(H^{(0)})$ and $\{ R_N^{(N+1)}, K^{(0)} \}\in Rg(K^{(0)})$.

\smallskip

We note the following fact, which derives from the Jacobi identity: if $f\in Ker(H^{(0)})$ then $\{ f, K^{(0)}\}\in Ker(H^{(0)})$.\\
Then we have that $\{ \Pi_{Ker(H^{(0)})} R_N^{(N+1)}, K^{(0)}  \}\in Ker(H^{(0)})$ and by \eqref{identita}
\[
\{ \Pi_{Ker(H^{(0)})} R_N^{(N+1)}, K^{(0)}  \}=-\{ \Pi_{Rg(H^{(0)})} R_N^{(N+1)}, K^{(0)}  \}+\{ H^{(0)}, Q_N^{(N+1)}\}\in Rg(H^{(0)}).
\]
Thus $\{ \Pi_{Ker(H^{(0)})} R_N^{(N+1)}, K^{(0)}  \}=0$ and 
\[
\Pi_{Ker(H^{(0)})} R_N^{(N+1)}=\Pi_{Ker(H^{(0)})} \Pi_{Ker(K^{(0)})} R_N^{(N+1)}.
\]
By symmetry $\Pi_{Ker(K^{(0)})} Q_N^{(N+1)}=\Pi_{Ker(H^{(0)})} \Pi_{Ker(K^{(0)})} Q_N^{(N+1)}$. Hence
\begin{equation}\label{prisoner}
\Pi_{Rg(H^{(0)})} \Pi_{Ker(K^{(0)})} Q_N^{(N+1)}=\Pi_{Rg(K^{(0)})} \Pi_{Ker(H^{(0)})} R_N^{(N+1)}=0.
\end{equation}
In order to obtain the Birkhoff normal form at order $N+1$ we consider a Birkhoff transformation $\Phi_{\chi_{n+1}}$ with generator $\chi_{N+1}\in \mathscr{P}^{(N+1)}$ and we define $\Phi_{N+1}:=\Phi_N\circ \Phi_{\chi_{n+1}}$. The function $\chi_{N+1}$ is chosen in order to solve the homological equation
\[
\{ H^{(0)}, \chi_{N+1}\}=-\Pi_{Rg(H^{(0)})} R_N^{(N+1)} \stackrel{(\ref{prisoner})}{=}-\Pi_{Rg(K^{(0)})}\Pi_{Rg(H^{(0)})} R_N^{(N+1)}.
\]
We now show that $\chi_{N+1}$ solves also the homological equation for $K_N$. Indeed, by the fact that $\mathrm{ad}_{H^{(0)}}^{-1}$ commutes with $\mathrm{ad}_{K^{(0)}}$ on the intersection $Rg(H^{(0)})\cap Rg(K^{(0)})$
\[
\{ K^{(0)}, \chi_{N+1}\}=-\mathrm{ad}_{H^{(0)}}^{-1}\{ K^{(0)}, \Pi_{Rg(K^{(0)})}\Pi_{Rg(H^{(0)})} R_N^{(N+1)} \}
\]
and by \eqref{identita}, \eqref{prisoner} we get
\[
\{ K^{(0)}, \Pi_{Rg(K^{(0)})}\Pi_{Rg(H^{(0)})} R_N^{(N+1)} \}=\{H^{(0)}, \Pi_{Rg(K^{(0)})}\Pi_{Rg(H^{(0)})} Q_N^{(N+1)}  \}.
\]
\end{proof}

Suppose that we performed $N$ Birkhoff steps. Then the transformed Hamiltonian is
\[
H_N=H^{(0)}+Z_N+R_N, \quad Z_N\in\mathscr{P}^{(\le N)}, R_N\in\mathscr{F}^{(\geq N+1)}
\]
where $Z_N$ is action-preserving.\\
If we perform the $(N+1)$-th step then the term $\Pi_{Ker(H^{(0)})}  R_N^{(N+1)}\in\mathscr{P}^{(N+1)}$ contributes to the normal form. We have to show the following claim

\smallskip

\begin{itemize}
\item \textit{The resonant term $\Pi_{Ker(H^{(0)})}  R_N^{(N+1)}$ is supported on $\mathcal{N}_{N+1}^*$.}
\end{itemize}

\smallskip

The identity \eqref{identita} holds for any $N\geq 0$. Thus if we consider the set of commuting Hamiltonians $K_1,\dots, K_{N+2}$ (defined in \eqref{DPHamiltonian} and in Theorem \ref{costantiMotoDP}), which are in the Birkhoff normal form after $N$ steps
\[
K_{m, N}=K_m^{(0)}+W_{m, N}+Q_{m, N} \qquad m=1, \dots, N+2,
\]
we have that $R^{(N+1)}_N\in\mathscr{P}^{(N+1)}$ satisfies
\begin{equation}
\{ H^{(0)}, Q^{(N+1)}_{m, N}\}=\{ K_m^{(0)}, R^{(N+1)}_N \}, \quad \forall m=1, \dots, N+2.
\end{equation}
Thus by \eqref{formacostanti} the above relation writes as (recall \eqref{Gdivisor})
\begin{equation}\label{base}
\Omega(\alpha)\, K_{m, N, \alpha}^{(N+1)}=\Omega_{K_m^{(0)}}(\alpha)\,R_{N, \alpha}^{(N+1)}=\Big(\sum_j \,j^{2 (m-1)}(1+j^2)^2\,\omega(j)\,\alpha_j\Big) R_{N, \alpha}^{(N+1)},
\end{equation}
for any $\ \alpha\in\mathcal{I}_{N+1}$ and any  $m=1, \dots, N+2$.

\noindent
If $\Omega(\alpha)\neq 0$ or $R_{N, \alpha}^{(N+1)}=0$ then the resonant term $\Pi_{Ker(H^{(0)})}  R_N^{(N+1)}=0$. \\
If $\Omega(\alpha)=0$ then $R_{N, \alpha}^{(N+1)}\neq 0$ if and only if
\begin{equation}\label{torrone}
\sum_j \,j^{2 (m-1)}(1+j^2)^2\,\omega(j)\,\alpha_j=\sum_j \,j^{2 m-1}(1+j^2)\,(4+j^2)\,\alpha_j=0, \quad \forall \alpha \in\mathcal{I}_{N+1}, \quad \forall m=1, \dots, N+2.
\end{equation}
%
We have to prove that the linear $(N+3)\times(N+3)$-dimensional system \eqref{torrone} has no solutions, except for $\alpha\in \mathcal{N}_{N+1}^*$.\\
Finding a solution of \eqref{torrone} is equivalent to prove that there are integers $j_1, \dots, j_{N+3}\in\mathbb{Z}\setminus\{0\}$ such that the following matrix $M$ has a non trivial kernel
\begin{equation}\label{torrone2}
 M:=\mathrm{diag}_{i=1, \dots, N+3} \Big( (1+j_i^2)\,(4+j_i^2) \Big) \,V, \quad V : =\begin{bmatrix}
j_1 & \dots & j_N\\
j_1^3 & \dots & j_N^3\\
\vdots & \dots & \vdots\\
j_1^{2 (N+3)+1}& \dots & j_N^{2(N+3)+1} 
\end{bmatrix}
.
\end{equation}
Hence our goal is to prove that $\det M=0$ if and only if $(j_1, \dots, j_{N+3})=(i, -i, j, -j, \dots)$ or its permutations. Clearly $\det M=0$ if and only if $\det V=0$.
Note that
\[
\det V=\Big(\prod_{i=1}^{N+3} j_i\Big)\,\,\,\det \begin{bmatrix}
1 & \dots & 1\\
j_1^2 & \dots & j_{N+3}^2\\
\vdots & \dots & \vdots\\
j_1^{2 N+6}& \dots & j_{N+3}^{2N+6} 
\end{bmatrix}
\]
and by renaming $x_i=j_i^2$ and by using the well known formula for the determinant of the Vandermonde matrix we have that
\[
\det V =\sqrt{\prod_{i=1}^{N+3} x_i}\,\,\,\prod_{i< j} (x_i-x_j).
\]
So it is clear that $\det V=0$ if and only if
$$\Big((j_1+j_2)\dots (j_{N+2}+j_{N+3})\Big)\Big( (j_1-j_2)\dots (j_{N+2}-j_{N+3}) \Big)=0.$$
By the fact that the indices $j_i$ are all distincts we deduce the claim.

%
%

\begin{appendix}

\section{Appendix } 

\subsection{Proof of local existence } \label{appendA}
Here we prove Proposition \ref{esistesol2} about the local well-posedness 
of Eq. \eqref{DP}. The argument follows closely the proof of Theorem 1.1 in 
\cite{HH}, which discusses the well-posedness of the dispersionless DP equation 
\eqref{dispersionless}. We present the proof in the compact case; during the proof we point out the minor changes one has to make in order to adjust the proof to the 
noncompact case.

First observe that if $u \in H^s(\mathbb{T};\mathbb{R})$, then 
$u \del_x u + (1-\del_{xx})^{-1} \del_{xxx}u \in H^{s-1}(\mathbb{T};\mathbb{R})$. 
We handle this problem by considering the mollified version of \eqref{DP}: fix a Schwartz function $j \in \mathcal{S}(\R)$ satysfying $0 \leq \hat{j}(\xi) \leq 1$ for all $\xi \in \mathbb{R}$, and $\hat{j}(\xi)=1$ for $|\xi|\leq 1$. Then we define the periodic functions 
$j_\epsilon$, $0 < \epsilon \leq 1$ by the following formula,
\begin{align*}
j_\epsilon(x) &:=\frac{1}{2\pi} \sum_{n \in \mathbb{Z}} \hat{j}(\epsilon n) e^{inx},
\end{align*}
and we define the mollifier by 
\begin{align}
J_\epsilon f &:= j_\epsilon \ast f.
\end{align}

By direct computation one has that $\hat{j}_\epsilon(k) =\hat{j}(\epsilon k)$; 
furthermore, for $0 < \sigma \leq s$ the map $Id-J_{\epsilon}:H^s(\mathbb{T};\mathbb{R}) \to H^\sigma(\mathbb{T};\mathbb{R})$ satisfies
\begin{align} \label{opnorm}
\| Id-J_{\epsilon} \|_{ L( H^s(\mathbb{T};\mathbb{R}) , H^\sigma(\mathbb{T};\mathbb{R}) ) } &= o(\epsilon^{s-\sigma}).
\end{align}

Now recall that Eq. \eqref{DP} is obtained by its dispersionless version 
\eqref{dispersionless} by applying the boost $u \mapsto \mathtt{c}+u$. 
This means that we can derive the mollified version of \eqref{DP} 
simply by translation form the mollified dispersionless DP equation 
(see Eq. (110) in \cite{HH}; in the rest of this section we denote by $D$ 
the operator $(1-\del_{xx})^{1/2}$)

\begin{align} \label{mollDP}
u_t &= - J_\epsilon (J_\epsilon (\mathtt{c}+u) \del_x J_\epsilon (\mathtt{c}+u)) -\frac{3}{2} \del_x D^{-2} \left( (\mathtt{c}+u)^2 \right), \; 0 < \epsilon \leq 1,
\end{align}
with initial datum $u(0,x)=u_0(x) \in H^s$. 

Now introduce the map $F_\epsilon: H^s \to H^s$,
\begin{align}
F_\epsilon(u) &= - J_\epsilon (J_\epsilon (\mathtt{c}+u) \del_x J_\epsilon (\mathtt{c}+u)) -\frac{3}{2} \del_x D^{-2} \left( (\mathtt{c}+u)^2 \right) \nonumber \\
&= -J_\epsilon (\mathtt{c} J_\epsilon \del_x J_\epsilon u) - J_\epsilon (J_\epsilon u \del_x J_\epsilon u) - \frac{3}{2} \del_x D^{-2}(u^2) - 3\mathtt{c} \del_x D^{-2} (u); \label{Feps}
\end{align}
we can observe that for any $\epsilon$ the map $F_\epsilon$ is differentiable.  
Therefore \eqref{mollDP} with initial datum $u(0,\cdot)=u_0 \in H^s$, $s>3/2$, 
defines an ODE on $H^s$, which thus admits a unique solution $u_\epsilon$ with 
existence time $T_\epsilon>0$. 
We now prove Proposition \ref{esistesol2} after some intermediate lemmata. 

\begin{lem} \label{LWPlemma1}
Let $s>3/2$, then there exists $C(s)>0$ such that the existence time $T_\epsilon$ 
for the solution of \eqref{mollDP} satisfies
\begin{align} \label{timespan}
T_\epsilon &\geq \bar{T}:=\bar{T}(\|u_0\|_{H^s})=\frac{1}{2C(s)\|u_0\|_{H^s}},
\end{align}
while the solution $u_\epsilon$ satisfies
\begin{align} \label{estsolmoll}
\|u_\epsilon(t)\|_{H^s} \leq 2 \|u_0\|_{H^s}, \; |t| \leq \bar{T}.
\end{align}
\end{lem}

\begin{proof}
First apply the operator $D^{s}$ to both sides of \eqref{mollDP}, 
and multiply both sides by $D^{s} u_\epsilon$. 
By integration and \eqref{Feps} we have
\begin{align}
\frac{1}{2} \frac{d}{dt} \|u_\epsilon\|^2_{H^s} &= - \int_0^{2\pi}  D^s u_\epsilon \; D^s \left[ J_\epsilon ( J_\epsilon u_\epsilon \del_x J_\epsilon u_\epsilon ) + \frac{3}{2} \del_x D^{-2} \left( u^2 \right) \right] dx,
\end{align}
where we exploited the fact that the first and the last term in \eqref{Feps} 
are linear in $u$ and are given by the action of skew-adjoint Fourier 
multipliers which commute with $D^s$ (all these facts imply that the terms one 
would need to add to formula (117) in \cite{HH} vanish).  \\
In order to commute the operator $D^s$ with $J_\epsilon u_\epsilon$ we apply the 
following Kato-Ponce commutator estimate (see \cite{KatoPonce}).
\begin{lem}
Let $s>0$, then there exists $C(s)>0$ such that
\begin{align} \label{KPineq}
\|D^s(fg)-f D^s g\|_{L^2} &\leq C(s) \left( \|D^s f\|_{L^2} \|g\|_{L^\infty} + \|\del_x f\|_{L^\infty} \|D^{s-1}g\|_{L^2} \right).
\end{align}
\end{lem}
By exploiting \eqref{KPineq} and Sobolev embedding we have that there exists 
$C(s)>0$ such that
\begin{align*}
\frac{d}{dt} \|u_\epsilon\|^2_{H^s}  \leq 2 C(s) \|u_\epsilon\|^3_{H^s},
\end{align*}
which implies that 
\begin{align} \label{ineq}
\| u_\epsilon(t) \|_{H^s} &\leq \frac{ \|u_0\|_{H^s} }{ 1- C(s) \|u_0\|_{H^s} t},
\end{align}
and by setting $\bar{T}:= \frac{1}{ 2 C(s) \|u_0\|_{H^s} }$ we get the thesis.
\end{proof}

\begin{lem} \label{LWPexist}
Consider Eq. \eqref{DP} with initial datum $u_0 \in H^s$, $s>3/2$. 
Then there exists a solution $u \in C([-\bar T, \bar T]; H^s)$ with $\bar T$ 
as in \eqref{timespan} such that 
\begin{align} \label{estsol}
\|u(t)\|_{H^s} \leq 2 \|u_0\|_{H^s}, \; |t| \leq \bar{T}.
\end{align}
\end{lem}

\begin{proof}
To simplify the notation we set $I:=[-\bar T, \bar T]$. The proof is divided in several steps, whose purpose is to obtain the convergence of the family 
$(u_\epsilon)_{0 < \epsilon \leq 1}$ by extracting subsequences $(u_{\epsilon_\nu})_\nu$; 
after each such extraction, we assume that the resuling sequence is relabeled 
as $(u_\epsilon)_\epsilon$. 

\vspace{0.8em}

{\bf Step 1: weak$^\star$ convergence in $L^\infty(I;H^s)$.}
The family $(u_\epsilon)_{0 < \epsilon \leq 1}$ is bounded in the space 
$C(I;H^s) \subset L^\infty(I;H^s)$. Since $L^\infty(I;H^s)$ is the dual of the 
space $L^1(I;H^s)$, Alaoglu's Theorem implies that $(u_\epsilon)_\epsilon$ is 
precompact with respect to the weak$^\star$ topology. Hence there exists a 
subsequence $(u_{\epsilon_\nu})_\nu$ which converges to $u \in L^\infty(I;H^s)$ 
weakly$^\star$, and such that $u$ satisfies \eqref{estsol}. \\
{\bf Step 2: convergence in $C(I;H^{s-1})$.} In order to show the strong convergence in $C(I;H^{s-1})$, we show that $(u_\epsilon)_\epsilon$ satisfies the hypotheses of Ascoli-Arzel\'a Theorem. Indeed, $(u_\epsilon)_\epsilon$ is equicontinuous, since for any $t_1, t_2 \in I$
\begin{align*}
\|u_\epsilon(t_1) - u_\epsilon(t_2)\|_{H^{s-1}} &\leq \sup_{t \in I} \| \del_t u_\epsilon \|_{H^{s-1}} \, |t_1-t_2| \nonumber \\
&\stackrel{ \eqref{mollDP},\eqref{Feps}, \eqref{KPineq} }{\leq} 10 C(s) (\|u_0\|_{H^s}+\|u_0\|^2_{H^s}) |t_1-t_2|.
\end{align*}
Setting $U(t):=(u_\epsilon(t))_\epsilon$, we see that for any $t \in I$ the set 
$U(t) \subset H^s$ is bounded. On the other hand, since $\mathbb{T}$ is a 
compact manifold, we have that the inclusion $i:H^s \to H^{s-1}$ is compact. 
Therefore $U(t)$ is precompact in $H^{s-1}$. \\
{\bf Step 3: convergence in $C(I;H^{s-\sigma})$, $\sigma \in (0,1)$.} For each 
$\sigma \in (0,1)$ we have 
\begin{align*}
\|u_\epsilon\|_{ C^\sigma(I;H^{s-\sigma}) } &= \sup_{t \in I} \|u_\epsilon\|_{ H^{s-\sigma} } + \sup_{t \neq t'} \frac{ \|u_\epsilon(t)-u_\epsilon(t')\|_{H^{s-\sigma}}  }{ |t-t'|^\sigma }.
\end{align*}
Now, the first term in the right-hand side of the above inequality is bounded, 
since 
\begin{align*}
\sup_{t \in I} \|u_\epsilon(t)\|_{H^{s-\sigma}} &\stackrel{ \eqref{estsolmoll} }{\leq} 2 \|u_0\|_{H^s},
\end{align*}
while the second term can be bounded by exploiting \eqref{Feps} and \eqref{KPineq}. Putting these bounds together allows us to apply Ascoli-Arzel\'a Theorem, 
since the equicontinuity condition follows form
\begin{align*}
\|u_\epsilon(t_1)-u_\epsilon(t_2)\|_{H^{s-\sigma}} &\leq \|u_\epsilon\|_{C^\sigma(I;H^{s-\sigma})} |t_1-t_2|^\sigma,
\end{align*}
while the precompactness condition can be verified as in the previous step.\\
{\bf Step 4: convergence in $C( I;C^1(\mathbb{T}) )$.}
Now fix $\sigma \in (0,1)$ such that $s-\sigma>3/2$, then by Sobolev embedding 
implies that $u_\epsilon \to u$ in $C(I;C^1(\mathbb{T}))$. Now we need to study 
$\del_t u_\epsilon$. Starting with the two non-local terms of \eqref{Feps}, the 
continuity of the operator $\del_x D^{-2}$ implies that 
$\del_x D^{-2}(u_\epsilon^2) \to \del_x D^{-2} (u^2)$ and 
$\del_x D^{-2} u_\epsilon \to \del_x D^{-2} u$ in $C(I;C(\mathbb{T}))$. 
To handle the first two terms, first observe that 
\begin{align} \label{estBurgterm}
\|J_\epsilon u_\epsilon - u \|_{ C(I;C(\mathbb{T})) } &\leq \|J_\epsilon u_\epsilon - u_\epsilon \|_{ C(I;C(\mathbb{T})) } + \| u_\epsilon - u \|_{ C(I;C(\mathbb{T})) }. 
\end{align}
To estimate the first term in the right-hand side of \eqref{estBurgterm}, 
choose $r \in (1/2,s)$, and observe that for any $t \in I$ \eqref{opnorm} 
implies that there exists $C(r) >0$ such that
\begin{align*}
\|J_\epsilon u_\epsilon - u_\epsilon \|_{ C(I;C(\mathbb{T})) } &\leq 2 C(r) \|Id-J_\epsilon\|_{L(H^s ; H^r)} \|u_0\|_{H^s} = o(\epsilon^{s-r}),
\end{align*}
from which we can deduce that $J_\epsilon u_\epsilon\to u$ in $C(I;C(\mathbb{T}))$. 
With a similar argument we can show that $J_\epsilon \del_x u_\epsilon \to \del_x u$ in $C(I;C(\mathbb{T}))$. Therefore one can conclude that
\begin{align*}
\del_t u_\epsilon &\to -(\mathtt{c}+u)\del_x(\mathtt{c}+u) -\frac{3}{2}\del_x D^{-2}( (\mathtt{c}+u)^2 )
\end{align*}
in $C(I;C(\mathbb{T}))$. 
Recalling that also $u_\epsilon \to u$ in $C(I;C^1(\mathbb{T}))$, we can deduce 
that $t \mapsto u(t)$ is a differentiable map such that
\begin{align*}
\del_t u &= -(\mathtt{c}+u)\del_x(\mathtt{c}+u) -\frac{3}{2}\del_x D^{-2}( (\mathtt{c}+u)^2 ).
\end{align*}
{\bf Step 5: convergence in $C(I;H^s)$.}
Fix $t \in I$ and take a sequence $(t_n)_{n \in \mathbb{N}} \to t$. 
Since $u \in L^\infty(I;H^s)$, we have that $t \mapsto u(t)$ is continuous with 
respect to the weak topology on $H^s$; thus, to verify the continuity we just 
need to check that the map $t \mapsto \|u(t)\|^2_{H^s}$ is continuous. We begin 
by introducing 
\begin{align*}
F(t) &:= \|u(t)\|^2_{H^s}, \\
F_\epsilon(t) &:= \|J_\epsilon u(t)\|^2_{H^s}.
\end{align*}
Now, \eqref{opnorm} implies that $F_\epsilon \to F$ pointwise as $\epsilon \to 0$. Therefore it suffices to show that each $F_\epsilon$ is Lipschtiz and that the Lipschitz constants for this family are bounded. Since
\begin{align} \label{HsFeps}
\frac{1}{2} F'_\epsilon(t) &= - \int_0^{2\pi}  D^s J_\epsilon u \; D^s J_\epsilon \left[(\mathtt{c}+u)\del_x(\mathtt{c}+u)\right] \, dx -\frac{3}{2} \int_0^{2\pi} D^s J_\epsilon u D^sJ_\epsilon \del_x D^{-2}( (\mathtt{c}+u)^2 ) \, dx.
\end{align}
To bound the first term on the right-hand side of \eqref{HsFeps} we need to use 
the commutator estimate \eqref{KPineq}, in order to commute the operator $D^s$ 
with $u$; but since we also need to commute $J_\epsilon$ with $u$, we exploit the 
following result (see \cite{Taylor}):

\begin{lem}
Let $f,g \in H^s$, then there exists $C>0$ such that 
\begin{align*}
\| [f,J_\epsilon] \del_x g\|_{L^2} &\leq C \|f\|_{C^1(\mathbb{T})} \|g\|_{L^2(\mathbb{T})}.
\end{align*}
\end{lem}

By applying the Cauchy-Schwarz inequality, the algebra property, and estimate 
\eqref{estnormsol} on the size of the solution, we can conclude that 
there exists $C(s)>0$ such that
\begin{align*}
|F'_\epsilon(t)| &= \left| \frac{d}{dt} \|J_\epsilon u(t)\|^2_{H^s} \right| \leq C(s).
\end{align*}

\end{proof}

\begin{remark}
To adjust the proof for the noncompact case, we have to define the mollifiers $J_\epsilon$ in the following way: first we fix $j \in \mathcal{S}(\mathbb{R})$ such that $\hat{j}(\xi)=1$ for $|\xi|\leq 1$. Then we set $j_\epsilon(x):= \epsilon^{-1} j(x/\epsilon)$; this gives again that $\|Id-J_\epsilon\|_{L(H^s,H^r)} =o(\epsilon^{s-r})$. \\
Moreover, we exploited the compactness of $\mathbb{T}$ in order to satisfy the hypotheses of Ascoli's theorem; in the noncompact case the embedding $H^s \to H^{s'}$ for $s>s'$ does not define a compact operator. We handle this problem by first fixing $\phi \in \mathcal{S}(\mathbb{R})$ with $0 < \phi(x) \leq 1$. Then Rellich's Theorem implies that the operator $u_\epsilon \mapsto \phi u_\epsilon$ is compact from $H^s$ to $H^{s'}$. Using this modification and by recalling that $\phi \neq 0$, we obtain again the existence of a solution $u$.
\end{remark}

\begin{lem} \label{LWPunique}
Consider Eq. \eqref{DP} with initial datum $u_0 \in H^s$, $s>3/2$. 
Then its solution $u \in C([-\bar T, \bar T]; H^s)$ with $\bar T$ 
as in \eqref{timespan} is unique.
\end{lem}

\begin{proof}
Let $u_0 \in H^s$, and let $u$ and $w$ be two solutions to \eqref{DP} with 
$u(0,\cdot)=w(0,\cdot)=u_0$. Consider $v:=u-w$, then
\begin{align} \label{eqdiff}
\del_t v &= -\frac{1}{2} \del_x \left[ (2\mathtt{c}+u+w)v \right] -\frac{3}{2} \del_x D^{-2} ((2\mathtt{c}+u+w)v).
\end{align}
Fix $\sigma \in (1/2,s-1)$; then
\begin{align} \label{normv}
\frac{d}{dt} \|v\|^2_{H^\sigma} &= -\int_0^{2\pi} D^\sigma v \left[  D^\sigma \del_x ((2\mathtt{c}+u+w)v) + 3 \del_x D^{\sigma-2}((2\mathtt{c}+u+w)v) \right] \, dx.
\end{align}

In order to bound the first term in the right-hand side of \eqref{normv} we 
commute $D^\sigma \del_x$ with $u+w$ by exploiting the following Calderon-Coifman-Meyer estimate (see Proposition 4.2 in \cite{TayCommEst})

\begin{lem}
Let $\sigma \geq -1$, then for any $\rho>3/2$ such that $\sigma+1 \leq \rho$ 
there exists $C>0$ such that 
\begin{align*}
\| [D^\sigma \del_x,f]v \|_{L^2} &\leq C \|f\|_{H^\rho} \|v\|_{H^\sigma}.
\end{align*}
\end{lem}

The nonlocal term is bounded by Plancherel and Cauchy-Schwarz inequality. 
Hence there exists $c(s)>0$ such that
\begin{align*}
\frac{d}{dt} \|v(t)\|^2_{H^\sigma} &\leq c(s) \|v\|^2_{H^\sigma}; \\
\|v\|_{H^\sigma} &\leq e^{c(s) \, T} \|v(0)\|_{H^\sigma}=0,
\end{align*}
and we can conclude that $u=w$.
\end{proof}

\begin{lem} \label{LWPcontdep}
Consider Eq. \eqref{DP} with initial datum $u_0 \in H^s$, $s>3/2$. 
Then the solution map from $H^s \to C(I;H^s)$ 
($I=[-\bar T,\bar T]$, with $\bar T$ as in \eqref{timespan}) 
given by $u_0 \mapsto u$ is continuous.
\end{lem}

\begin{proof}
Fix $u_0 \in H^s$, and let $(u_{0,n})_n \subset H^s$ be a sequence such that 
$\lim_{n \to \infty} u_{0,n} = u_0$. Then, if $u_n$ is the solution of 
Eq. \eqref{DP} with initial datum $u_{0,n}$, we want to show that
\begin{align} \label{contdepconv}
\lim_{n \to \infty} u_n &= u \; \; \text{in} \; \; C(I;H^s);
\end{align}
equivalently, let $\eta>0$, we want to show hat there exists $N>0$ such that 
\begin{align} \label{contdepconv2}
\|u-u_n\|_{C(I;H^s)} &< \eta, \; \; \forall n > N.
\end{align}
As before, we will ue the convolution operator to smooth out the initial data. 
Let $0 < \epsilon \leq 1$, let $u^\epsilon$ be the solution to \eqref{DP} 
with initial datum $J_\epsilon u_0 = j_\epsilon \ast u_0$ and let $u^\epsilon_n$ 
be the solution of \eqref{DP} with initial datum $J_\epsilon u_{0,n}$. Then
\begin{align} \label{distsol}
\|u-u_n\|_{C(I;H^s)} &\leq \|u-u^\epsilon\|_{C(I;H^s)} + \|u^\epsilon-u_n^\epsilon\|_{C(I;H^s)} + \|u^\epsilon-u_n\|_{C(I;H^s)}.
\end{align}
We will prove that each of these terms can be bounded by $\eta/3$, 
for suitable choices of $\epsilon$ and $N$. We also point out that the quantity 
$\epsilon$ will be independent of $N$ and will only depend on $\eta$, while 
the choice of $N$ will depend on both $\eta$ and $\epsilon$. 

\vspace{0.8em}

We start with $\|u^\epsilon-u_n^\epsilon\|_{C(I;H^s)}$. 
Set $v:=u^\epsilon-u_n^\epsilon$, then $v$ satisfies
\begin{align*}
\del_t v &= -\frac{1}{2} \del_x \left[ (2\mathtt{c}+u^\epsilon+u_n^\epsilon)v \right] -\frac{3}{2} \del_x D^{-2} ((2\mathtt{c}+u^\epsilon+u_n^\epsilon)v), \nonumber \\
v(0) &= u^\epsilon(0)-u_n^\epsilon(0) = J_\epsilon u_0 -J_\epsilon u_{0,n},
\end{align*}
and 
\begin{align} \label{Hsnormv}
\frac{1}{2} \frac{d}{dt} \|v\|^2_{H^s} &= - \int_0^{2\pi} D^s v \, D^s \left[ \frac{1}{2} \del_x \left[ (2\mathtt{c}+u^\epsilon+u_n^\epsilon)v \right] +\frac{3}{2} \del_x D^{-2} ((2\mathtt{c}+u^\epsilon+u_n^\epsilon)v) \right] dx.
\end{align}
Applying \eqref{KPineq} and the estimate $\|u^\epsilon\|_{H^{s+1}} \leq C/\epsilon$, \eqref{Hsnormv} implies that there exists $c_s>0$ such that
\begin{align} \label{ineqdiff}
\frac{1}{2} \frac{d}{dt} \|v(t)\|^2_{H^s} &\leq \frac{c_s}{\epsilon} \|v(t)\|^2_{H^s},
\end{align}
which in turn leads to
\begin{align} \label{ineqdiff2}
\|v(t)\|_{H^s} &\leq e^{c_s T/\epsilon} \|v(0)\|_{H^s} \; \leq \; 2 e^{c_sT/\epsilon} \|u_0-u_{0,n}\|_{H^s}.
\end{align}
Notice that \eqref{ineqdiff2} does not imply any constraint on $\epsilon$; 
however, handling the first and the third term in the right-hand side of 
\eqref{distsol} will require $\epsilon$ to be small. 
After chosing $\epsilon$, we will take $N$ so large that 
$\|u_0-u_{0,n}\|_{H^s} < \frac{\eta}{6} e^{-c_sT/\epsilon}$, which will imply that 
$\|u^\epsilon-u_n^\epsilon\|_{C(I;H^s)} < \eta/3$.

\vspace{0.8em}

Now we estimate $\|u^\epsilon-u\|_{C(I;H^s)}$ and $\|u^\epsilon-u_n\|_{C(I;H^s)}$. 
We set $v:=u^\epsilon-u$ and $v_n:=u_n^\epsilon-u_n$. Since $v$ and $v_n$ will 
satisfy the same energy estimates, we will write $v_{(n)}$ to mean that an 
equation holds both with and without the subscript. 
We observe that $v_{(n)}$ solves the Cauchy problem
\begin{align*}
\del_t v_{(n)} &= -\frac{1}{2} \del_x \left[ (2\mathtt{c}+u^\epsilon+u_{(n)})v_{(n)} \right] -\frac{3}{2} \del_x D^{-2} ((2\mathtt{c}+u^\epsilon+u_{(n)})v_{(n)} ) \\
&= -\frac{1}{2} \del_x \left[ (2\mathtt{c}+2 u^\epsilon+v_{(n)})v_{(n)} \right] -\frac{3}{2} \del_x D^{-2} ((2\mathtt{c}+2 u^\epsilon+v_{(n)})v_{(n)} ), \nonumber \\
v(0) &= j_\epsilon \ast u_{0,(n)}-u_{0,(n)}.
\end{align*}
By exploiting \eqref{KPineq}, the Cauchy-Schwarz inequality and Sobolev embedding we get
\begin{align} \label{Hsnormvn}
\frac{1}{2} \frac{d}{dt} \|v_{(n)}(t)\|_{H^s} &\leq c'_s \left[ \|v_{(n)}\|^3_{H^s} + (1+\|u^\epsilon_{(n)}\|_{H^s}) \|v_{(n)}\|^2_{H^s} + (1+\|u^\epsilon_{(n)}\|_{H^{s+1}}) \|v_{(n)}\|_{H^{s-1}} \|v_{(n)}\|_{H^s} \right]
\end{align}
for some $c'_s>0$. 
Since $\|u^\epsilon_{(n)}(t)\|_{H^{s+1}} \leq c_1(s)/\epsilon$ and that 
$\|v_{(n)}(t)\|_{L^2}=o(\epsilon)$, \eqref{Hsnormvn} gives
\begin{align} \label{diffineq3}
\frac{d y}{d t} &\leq c_2(s) \left( y^2+y+\delta \right),
\end{align}
where $\delta=\delta(\epsilon) \to 0$ as $\epsilon \to 0$. 

The quadratic expression $y^2+y+\delta$ has roots
\begin{align} \label{roots}
r_{-1} = \frac{-1-\sqrt{1-4\delta}}{2}, \; &\; r_{0} = \frac{-1+\sqrt{1-4\delta}}{2}.
\end{align}
Restricting $\epsilon$ so that the roots given in \eqref{roots} are real-valued, we observe that $r_0$ and $r_{-1}$ are negative and, as $\delta \to 0$, we have 
$r_{-1} \to -1$ and $r_0 \to 0$. Setting $R := \sqrt{1 - 4\delta}$ and taking into account the constant $c_s$, we solve \eqref{diffineq3} via
\begin{align} \label{diffineqsol}
\frac{y(t)-r_0}{y(t)-r_{-1}} &\leq \gamma, \\
\gamma &:= e^{c_s RT} \frac{y(0)-r_0}{y(0)-r_{-1}}.
\end{align}
From here, we will treat the cases $y = \|v\|_{H^s}$ and $y = \|v_n\|_{H^s}$ separately. 

\vspace{0.8em}

\emph{Case $y = \|v\|_{H^s}$.} Using \eqref{opnorm} we have $y(0) \to 0$ as 
$\epsilon \to 0$. This implies that $\gamma \to 0$ as $\epsilon \to 0$. 
From \eqref{diffineqsol}, we then obtain $y(t) \leq y(t)-r_0 \leq \gamma [y(t)-r_{-1}]$. Solving for $y(t)$ gives us
\begin{align*}
y(t) &\leq \frac{-r_{-1}}{1-\gamma} \gamma \\
&\stackrel{\gamma \to 0}{ \to } 0.
\end{align*}
Therefore, for sufficiently small $\epsilon$ we can bound the first term of 
\eqref{distsol} by $\eta/3$. 

\vspace{0.8em}

\emph{Case $y = \|v_n\|_{H^s}$.} We begin by bounding $y(0)$ by
\begin{align*}
\|j_\epsilon \ast u_{0,n}-u_{0,n} \|_{H^s} &\leq 2 \|u_{0,n}-u_0\|_{H^s} + \|j_\epsilon \ast u_0-u_0\|_{H^s},
\end{align*}
which implies that
\begin{align}
\gamma &\leq \frac{e^{c_S RT}}{-r_{-1}} \left( 2 \|u_{0,n}-u_0\|_{H^s}+\|j_\epsilon \ast u_0-u_0\|_{H^s} \right) + \frac{r_0 e^{c_s RT}}{r_{-1}},
\end{align}
where we may independently choose $\epsilon$ sufficiently 
small and $N$ sufficiently large so that $\gamma < 1/2$. 
Then, arguing as in the previous case we obtain $y(t) \leq 2\gamma$. 
We may now further refine the choice of $\epsilon$ and $N$ so that 
$y(t)<\eta/3$, completing this case. Collecting our results completes the proof.
\end{proof}

\begin{remark}
The proofs for uniqueness of the solution and for the continuous dependence on the initial datum do not rely on compactness properties, hence they do not nedd any adjustment in the noncompact case.
\end{remark}

\subsection{Analyticity on Sobolev spaces}

We recall some facts about analytic functions on Banach spaces following Appendix A of \cite{PoTr}.

\begin{defi}{(\bf{Weakly analyticity})}
Let $E, F$ two complex Banach spaces and $U$ an open subset of $E$. A map $f\colon U\to F$ is said weakly analytic if for each $w\in U$, $h\in E$ and $L\in F^*$ the function
\[
z\mapsto L f(w+z h)
\]
is analytic in some neighborhood of the origin in $\mathbb{C}$ in the usual sense of one complex variable.
\end{defi}
\begin{teor}\label{TruboTeo}
Let $f\colon U\to F$ be a map from an open subset $U$ of a complex Banach space $E$ into a complex Banach space $F$. Then the following three statements are equivalent.
\begin{enumerate}
\item $f$ is analytic in $U$.
\item $f$ is locally bounded and weakly analytic in $U$.
\item $f$ is infinitely often differentiable on $U$, and is represented by its Taylor series in a neighborhood of each point in $U$.
\end{enumerate}
\end{teor}

\end{appendix}


\begin{thebibliography}{9}

%
%
%
%
%
%
%
%
%
%
%
%
%
\bibitem{Bambu} Bambusi D., \textit{Birkhoff normal form for some nonlinear PDEs}, Comm. Math. Physics 234 (2003).
%
\bibitem{BG} Bambusi D., Gr\'ebert B., \textit{Birkhoff normal form for partial differential equations with tame modulus}, Duke Math. J. (2006).


\bibitem{BonaSmith} Bona J., Smith R., \textit{The initial-value problem for the Korteweg-de Vries equation}, Philos. Trans. Roy. Soc. London Ser. A, 278(1287):555--601, 1975.


%
%
%
%
%
%
%
%
%
%
%
%
%
%
%
%
%
%
%
%
%
%
\bibitem{BertiDelort} Berti M., Delort J. M., \textit{Almost global existence of solutions for capillarity-gravity water waves equations with periodic spatial boundary conditions}, preprint, arXiv:1702.04674 (2017).
%
%
%
%
%
%
%
%
%
%
%
%
%
%
%
%
%

\bibitem{ConstantinEscher} Constantin. A., Escher J., \textit{Wave breaking for nonlinear nonlocal shallow water equations}, Acta Math. 181 (1998), no. 2, 229--243. doi:10.1007/BF02392586. https://projecteuclid.org/euclid.acta/1485891179


\bibitem{ConstLannes} Constantin A., Lannes D., \textit{The Hydrodynamical Relevance of the Camassa--Holm and Degasperis--Procesi Equations}, Arch. Rational Mech. Anal., Volume 192, Number 1, pp. 165--186  (2009).

\bibitem{CH} Camassa R., Holm D., \textit{An integrable shallow water equation with peaked solution}, Phys. Rev. Lett., Volume 71, Issue 11, pp. 1661--1664 (1993).

\bibitem{globbreak} Constantin A., \textit{Global existence and breaking waves for a shallow water equation: a geometric approach}, Ann. Inst. Fourier (Grenoble), Volume 50, pp. 321--362 (2000).

\bibitem{stabpeak} Constantin A., Strauss W.A., \textit{Stability of peakons}, Comm. Pure Appl. Math., Volume 53, Number 5, pp. 603--610 (2000).

\bibitem{Coclite} Coclite G.M., Karlsen K.H., \textit{On the well-posedness of the Degasperis--Procesi equation}, J. Funct. Anal., Volume 233, Number 1, pp. 60--91 (2006).

%



%
%
%
%
%
%
%

\bibitem{CraigBF} Craig W., Worfolk P., \textit{An integrable normal form for water waves in infinite depth}, Phys. D 84, no. 3-4, 513-531 (1995).

\bibitem{Deg} Degasperis A., Holm D. D., Hone A. N. W., \textit{A new integrable Equation with Peakon Solutions}, Theoretical and Mathematical Physics, Volume 133, Issue 2, pp 1463--1474 (2002).

\bibitem{DegPro} Degasperis A., Procesi M., \textit{Asymptotic Integrability}, in Symmetry and Perturbation Theory (A. Degasperis and G. Gaeta, eds.), World Scientific Publishing, 23--37 (1999).

\bibitem{scolar} Degasperis A., Procesi M., \textit{Degasperis-Procesi equation}, Scholarpedia, 4(2):7318 (2009).

%
%
%
\bibitem{Delort-2009}
Delort. J. M.,
\newblock {A} quasi-linear {B}irkhoff normal forms method. {A}pplication to the
  quasi-linear {K}lein-{G}ordon equation on $\mathds{S}^1$.
\newblock {\em Ast\'erisque}, 341, 2012.

\bibitem{Delort-sphere}
Delort. J. M.,
\newblock {\em {Q}uasi-{L}inear {P}erturbations of {H}amiltonian
  {K}lein-{G}ordon {E}quations on {S}pheres}.
\newblock American Mathematical Society, 2015.


\bibitem{DelortSzeft2} Delort J. M., Szeftel J., \textit{Long-time existence for semi-linear  Klein--Gordon equations with small cauchy data on Zoll manifolds}, Amer. J. Math. 128 (2006).
\bibitem{DelortSzeft1} Delort J. M., Szeftel J., \textit{Long-time existence for small data nonlinear Klein--Gordon equations on tori and spheres}, Internat. Math. Res. Notices 37 (2004).
%
%
%
%
%
%
%
%
%
%
%
%
%
%


%
%
%
\bibitem{Escher} Escher, J., Liu, Y., Yin, Z., \textit{Global weak solutions and blow-up structure for the Degasperis--Procesi equation}, J. Funct. Anal., Volume 241, Issue 2, pp. 457--485 (2006). doi:10.1016/j.jfa.2006.03.022


\bibitem{FGP}  Feola R., Giuliani F., Procesi M., \textit{Quasi-periodic solutions for Hamiltonian perturbations of Degasperis-Procesi equation}. in preparation.

\bibitem{Luca} Genovese G., Luc\'a R., Valeri D., \textit{Gibbs measures associated to the integrals of motion of the periodic derivative nonlinear Schr\"odinger equation} , arXiv:1502.05967.

\bibitem{HH} Himonas, A., Holliman, C., \textit{On well--posedness of the Degasperis--Procesi equation}, Discrete Contin. Dyn. Syst, Volume 31, Issue 2, pp. 469--488  (2011).

\bibitem{Ho} Holm D.D., Staley M.F., \textit{Wave structure and nonlinear balances in a family of evolutionary PDEs}, SIAM J. Appl. Dyn. Syst., Volume 2, Issue 3, pp. 323--380 (2003).



\bibitem{KatoPonce} Kato T., Ponce G., \textit{Commutator estimates and the Euler and Navier--Stokes equations}, Communications on Pure and Applied Mathematics, Volume 41, Issue 7, pp. 891--907 (1988).

\bibitem{Li} Li P., Olver P., \textit{Well-posedness and blow--up solutions for an integrable nonlinearly dispersive model wave equation}, J. Diff. Eq., Volume 162, Issue 1, pp. 27--63 (2000).

\bibitem{LiuYin} Liu Y., Yin Z., \textit{Global existence and blow--up phenomena for the Degasperis--Procesi equation}, Commun. Math. Phys., Volume 267, Issue 3, pp. 801--820 (2006). https://doi.org/10.1007/s00220-006-0082-5

\bibitem{Liu} Liu Y, \textit{Global existence and blow--up solutions for a nonlinear shallow water equation}, Math. Ann., Volume 335, Number 3, pp. 717--735 (2006).

\bibitem{Lund} Lundmark H., Szmigielski J., \textit{Multi--peakon solutions of the Degasperis--Procesi equation}, Inverse Problems, Volume 19, Number 6, pp. 1241--1245 (2003).

%
%
%
%
%
%
%
%
%
%
%
%
%
%
%
%
%
\bibitem{KdVeKAM} Kappeler T., P\"oschel J., \textit{KAM and KdV}, Springer (2003).
%
%
%
%
%
%
%

\bibitem{Matsuno} Matsuno Y., \textit{Multisoliton solutions of the Degasperis--Procesi equation and their peakon limit}, Inverse Problems, Volume 21, Number 5, pp. 1553--1570 (2005).

%
%
%
%
%
%
%
%


\bibitem{PoTr} P\"oschel J., Trubowitz E., \textit{Inverse Spectral Theory}, Volume 130, Academic Press, Orlando, 1987.

%
%
%
%
%
%
%
%
%
%
%
%


\bibitem{Taylor} Taylor M. E., \textit{Pseudodifferential Operators and Nonlinear PDEs}, Progress in Mathematics, Birkh\"auser (1991).

\bibitem{TayCommEst} Taylor, M. E., \textit{Commutator estimates}, Proceedings of the American Mathematical Society, Volume 131, Issue 5, pp. 1501-1507 (2003).

\bibitem{Tao} Tao T., \textit{Nonlinear Dispersive Equations: Local and Global Analysis}, CBMS Regional Conference Series in Mathematics, Volume 106,  pp 373 (2006).

\bibitem{Visciglia} Tzvetkov N., Visciglia N.,  \textit{Invariant measures and long time behaviour for the Benjamin-Ono equation}, Int. Math. Res. Not. 2014 17, 4679.


\bibitem{Visciglia2} Tzvetkov N., Visciglia N.,  \textit{Invariant measures and long time behaviour for the Benjamin-Ono equation II},  Journal de Math\'ematiques Pures et Appliqu\'ees, 103 (2014), 102.

\bibitem{Vak} Vakhnenko V.O., Parkes E.J., \textit{Periodic and solitary-wave solutions of the Degasperis--Procesi equation}, Chaos Solitons Fractals, Volume 20, Number 5, pp. 1059--1073 (2004).






\bibitem{Wu} Xinglong W., \textit{On the Cauchy problem for the periodic generalized Degasperis--Procesi equation}, J. Func. An., Volume 260, Issue 5, pp. 1428--1445 (2011). https://doi.org/10.1016/j.jfa.2010.10.014.

\bibitem{YinR} Yin Z., \textit{On the Cauchy problem for an integrable equation with peakon solutions}, Illinois Journal of Mathematics, Volume 47, Number 3, pp. 649--666 (2003).

\bibitem{YinT} Yin Z., \textit{Global existence for a new periodic integrable equation}, Journal of Mathematical Analysis and Applications, Volume 283, Issue 1, pp. 129--139 (2003).









\end{thebibliography}
\end{document}